\documentclass[abstracton, paper=a4, fontsize=11pt,bibliography=totoc,final]{scrartcl}

\pdfoutput=1

\usepackage[utf8]{inputenc}
\usepackage[T1]{fontenc}
\usepackage{lmodern}
\usepackage[french,english]{babel}
\usepackage{csquotes}
\usepackage{amsthm, amssymb, amsmath, amsfonts, mathrsfs, dsfont, esint,textcomp}
\usepackage[colorlinks=true, pdfstartview=FitV, linkcolor=blue, citecolor=blue, urlcolor=blue,pagebackref=false]{hyperref}
\usepackage{mathtools}
\usepackage{enumitem}

\usepackage[a4paper,left=20mm,right=20mm,top=30mm,bottom=30mm,marginpar=20mm]{geometry}

\usepackage{microtype}

\usepackage[backend=biber,bibencoding=utf8,giveninits,maxnames=4,style=alphabetic,isbn=false, doi=false, eprint= true, url= false]{biblatex}

\bibliography{Hoefer.bib}

\usepackage[notcite,notref,color]{showkeys}
\definecolor{labelkey}{gray}{.8}
\definecolor{refkey}{gray}{.8}

\definecolor{darkblue}{rgb}{0,0,0.7} 
\definecolor{darkred}{rgb}{0.9,0.1,0.1}
\definecolor{darkgreen}{rgb}{0,0.5,0}
\newcommand{\rh}[1]{{\color{darkgreen}{#1}}}

\newtheorem{thm}{Theorem}[section]
\newtheorem{prop}[thm]{Proposition}
\newtheorem{lem}[thm]{Lemma}
\newtheorem{cor}[thm]{Corollary}
\theoremstyle{remark}
\newtheorem{rem}[thm]{Remark}
\theoremstyle{definition}

\renewcommand{\leq}{\leqslant}
\renewcommand{\geq}{\geqslant}

\renewcommand{\subset}{\subseteq}

\newcommand{\E}{\mathbb{E}}

\newcommand{\J}{\mathfrak{J}}

\newcommand{\F}{\mathcal{F}}
\renewcommand{\L}{\mathscr{L}}

\newcommand{\N}{\mathbb{N}}

\newcommand{\1}{\mathbf{1}}
\newcommand{\R}{\mathbb{R}}

\renewcommand{\P}{\mathbb{P}}

\newcommand{\eps}{\varepsilon}

\renewcommand{\div}{\mathop{\mathrm{div}}}

\newcommand{\de}{\mathrm{d}}
\newcommand{\wto}{\rightharpoonup}

\newcommand{\dd}{\, \mathrm{d}}
\renewcommand{\H}{\mathcal H}
\newcommand{\loc}{\mathrm{loc}}
\newcommand{\Ocal}{\mathcal{O}}
\newcommand{\Fcal}{\mathcal{F}}
\newcommand{\Rcal}{\mathcal{R}}

\newcommand{\abs}[1]{\left|#1\right|}
\newcommand{\norm}[1]{\left\|#1\right\|}

\DeclareMathOperator*{\dist}{dist}
\DeclareMathOperator*{\dv}{div}

\DeclareMathOperator*{\supp}{supp}

\numberwithin{equation}{section}

\mathtoolsset{showonlyrefs}

\usepackage{authblk}

\setkomafont{date}{\large}

\author[1]{Richard M. H\"ofer\thanks{richard.hoefer@ur.de}}
\author[2]{Jonas Jansen\thanks{jonas.jansen@math.lth.se}}
\affil[1]{Faculty for Mathematics, Regensburg University, Germany}
\affil[2]{Centre for Mathematical Sciences, Lund University, Sweden}

\begin{document}
\title{Convergence rates and fluctuations for the Stokes-Brinkman equations as homogenization limit in perforated domains}

%
\maketitle




\begin{abstract}
We study the  homogenization of the Dirichlet problem for the Stokes equations in $\R^3$ perforated by $m$ spherical particles. 
We assume the positions and velocities of the particles to be identically and independently distributed random variables.
In the critical regime, when the radii of the particles are of order $m^{-1}$, the
homogenization limit $u$  is given as the solution to the Brinkman equations.
We provide optimal rates for the convergence $u_m \to u$ in $L^2$, namely $m^{-\beta}$ for all $\beta < 1/2$. Moreover, we consider the fluctuations. In the central limit scaling, we show that these converge to a Gaussian field, locally in $L^2(\R^3)$,
with an explicit covariance. 

Our analysis is based on  explicit approximations for the solutions $u_m$ in terms of $u$ as well as the particle positions and their velocities. These are shown to be accurate in $\dot H^1(\R^3)$ to order $m^{-\beta}$ for all $\beta < 1$.
Our results also apply to the analogous problem regarding the homogenization of the Poisson equations.
\end{abstract}

%
%
%
%
%
%
%
%
\tableofcontents

\section{Introduction}

Numerous applications regarding the dynamics of suspensions and aerosols call for macro- and mesoscopic models which couple the particle evolution to the fluid. One of the most well-known models are the so-called Vlasov-Navier-Stokes equations for spherical, non-Brownian inertial particles. If the fluid inertia is neglected, they reduce to the so-called Vlasov-Stokes equations which take the dimensionless form
\begin{align} \label{eq:Vlasov.Stokes}
\left\{\begin{array}{rl}
	\partial_t f + v \cdot \nabla_x f + \dv((u-v)f) = 0, \\
	- \Delta u + \nabla p +  \rho u - j = h, \quad \dv u = 0, \\
	\rho = \int f \dd v,  \quad j = \int v f \dd v,
	\end{array}\right.
\end{align}
where $f(t,x,v)$ is the particle density and $h$ is some external force acting on the fluid. 
For questions regarding modeling and applications of this system, we refer the reader to \cite{BoudinGrandmontLorzMoussa15} and the references therein.

The rigorous derivation of these equations from a microscopic system is a wide open problem. The main difficulty lies in the nature of the interaction of the particles which is only implicitly given through the fluid. Moreover it is singular and long range. 
A natural preliminary step towards the rigorous derivation of the Vlasov(-Navier)-Stokes equations consists in the derivation of the limit fluid equations in \eqref{eq:Vlasov.Stokes} without taking into account the particle evolution. These are the so-called Brinkman equations. The additional term $\rho u - j$ describes the effective drag force that the particles exert on the fluid: The drag force of a single particle in a Stokes flow is given by 
\begin{align}
	F_i = 6 \pi R (V_i - u_i),
\end{align}
where $R$ is the particle radius, $V_i$ its velocity and $u_i$ is the unperturbed fluid velocity at the position of the particle. Therefore, the total drag will be of order one if the number of particles $m$ (in a finite volume) times their radius $R_m$ is of order one. By making the convenient choice
\begin{align} \label{eq:R_m.Stokes}
	R_m = \frac 1 {6 \pi m},
\end{align}
the Brinkman equations in the form above arise based on a superposition principle for the drag forces.

The rigorous derivation of the Brinkman equations has attracted increasing attention over the last years,
with results both in the cases of zero  and non-zero  particle velocities, see e.g. \cite{Allaire90a,  GiuntiHoefer19, Gerard-Varet19} and \cite{DesvillettesGolseRicci08, HillairetMoussaSueur19,  CarrapatosoHillairet20}, respectively. The most recent results focus on the derivation under very mild assumptions for (random) particle configurations.  
Such investigations seem compulsory in order to eventually accomplish the rigorous derivation of the Vlasov(-Navier)-Stokes equations. 
In this regard, it is also desirable  to develop very accurate explicit approximations  for the 
microscopic solution  $u_m$ and to characterize its convergence rate to the limit $u$ as well as the associated fluctuations. In our paper, we focus on these aspects.

\subsection{Statement of the main result}

We  consider the perforated domain
	\[\Omega_m = \R^3\setminus \bigcup_{i=1}^{m} \overline{B}_i,\]
where the particles are given by \(B_i = B_{R_m}(X_i)\) with $R_m$ as in \eqref{eq:R_m.Stokes}. The particle  positions \(X_1,\ldots,X_m\) as well as their velocities \(V_1,\ldots,V_m\)  are random variables in \(\R^3\). For $h \in  \dot H^{-1}(\R^3)$, we  study the solution $u_m$ to the Stokes equations
\begin{align}\label{eq: CM}
\left\{\begin{array}{rl}
	-\Delta u_{m} + \nabla p_m  =  h, \quad \div u_m = 0 \quad &\text{in } \Omega_{m}, \\
	u_{m}  = V_i \quad &\text{in }  B_i, \: i=1,\ldots,m.
\end{array} \right.
\end{align}
We consider the case when \(Z_i = (X_i,V_i)\) are i.i.d. according to \(f\in\mathcal{P}(\R^3\times\R^3)\). We impose the following hypotheses on $f$:
\begin{enumerate}[label*=(H\arabic*)]
\item \label{ass:energy} \(\int_{\R^3\times \R^3} |v|^2 f(\de x, \de v) <\infty\);
	\item \label{ass:rho} the distribution of the centers \(\rho(\cdot) := \int_{\R^3} f(\cdot,\de v)\in W^{1,\infty}(\R^3)\) is compactly supported;
	\item \label{ass:j} \(j(\cdot) := \int_{\R^3} v f(\cdot,\de v) \in{H}^1(\R^3)\).
\end{enumerate}

We remark that we in particular allow to choose \(f(\de x, \de v) = \rho(x) \dd x \delta_{v=0}\) which means that all particle velocities are zero.

%

We note that it is classical that the Stokes equations \eqref{eq: CM} are  well-posed if the particles do not overlap in the sense that there exists a unique weak solution $u_m \in \dot H^{1}(\R^3)$.  As stated in the following lemma overlapping of particles  does not occur with probability approaching \(1\) as \(m\to\infty\).
The lemma is a standard result that can for example be found in \cite[Proposition A.3]{Hauray09}.
\begin{lem}\label{lem:prob.distances}
For \(\nu \geq 0 \), $L > 0$ let 	
\[\Ocal_{m,\nu,L} = \left\{(Z_i)_{i=1}^{m} = ((X_i,V_i))_{i=1}^{m} : \min\limits_{i\neq j}|X_i-X_j| > L m^\nu R_m \right\}.\]
Then, for all $0 \leq \nu < 1/3$ and all $L > 0$, there exists $m_0 > 0$ such that for all $m \geq m_0$
	\begin{align}
		\P(\Ocal^c_{m,\nu,L}) \leq C L m^{\nu - 1/3},
	\end{align}
	where $C$ depends only on $\rho$.
\end{lem}

For overcoming the problem of the ill-posedness of \eqref{eq: CM} for overlapping particles, we could restrict ourselves to configurations of non-overlapping particles. However, this results in the loss of the independence of the particle positions. Thus, for technical reasons, we prefer to define $u_m$ to be the solution to \eqref{eq: CM} for \((Z_i)_{i=1}^{m} \in\Ocal_{m,0,2}\) and $u_m = u$ for \((Z_i)_{i=1}^{m} \notin\Ocal_{m,0,2}\).


For the statement of our main result, we introduce $u \in \dot H^1(\R^3)$ as the unique weak solution to the Brinkman equations
\begin{equation}\label{eq: CM_hom}
-\Delta u +  (\rho u - j)+ \nabla p   =  h, \quad \dv u = 0 \quad \text{in } \R^3. \\
\end{equation}
We remark that this problem is  well-posed  due to the assumptions \ref{ass:rho} and \ref{ass:j} (note that $j$ also has compact support and hence $j \in \dot H^{-1}(\R^3)$). 

Moreover, we introduce the solution operator $A$  for the Brinkman equations with vanishing flux $j$.
More precisely, $A$,  which depends on $\rho$,  maps $g$ to to the solution $w$ of the equation
\begin{align} \label{eq: CM_lambda_hom}
-\Delta w + \rho w + \nabla p = g, \quad \div w = 0 \quad \text{in } \R^3.
\end{align}
	
\begin{thm} \label{th:main}
	Let $h \in \dot H^{-1}(\R^3)$ and let $u_m \in \dot H^1(\R^3)$ and $u \in \dot H^1(\R^3)$ be the unique weak solutions to \eqref{eq: CM} and \eqref{eq: CM_hom}.
	\begin{enumerate}[label*=(\roman*)]
		\item \label{it:L^2.Convergence} For any $\beta < 1/2$ and any compact set $K \subset \R^3$
		\begin{align}
		m^\beta \|u_m - u\|_{L^2(K)} \longrightarrow 0 \quad \text{ in probability}.
		\end{align}	
		
		\item \label{it:fluctuations} For every $g\in L^2(\R^3)$ with compact support, 
		\begin{align}
		m^{1/2} (g, u_m - u) \longrightarrow \xi[g]
		\end{align}
		in distribution, where $\xi$ is a Gaussian field with mean zero and covariance
		\begin{align} \label{eq:covaranaice}
		\begin{aligned}
		\E [\xi[g_1] \xi[g_2]] & =  \int_{\R^3\times\R^3}  \bigl((u(x)-v) \cdot (A g_1)(x)\bigr) \bigl((u(x)-v) \cdot (A g_2)(x)\bigr) f(\de x,\de v) \\
		& -  \left(\rho u  -j,A g_1\right)_{L^2}\left(\rho u -j,Ag_2\right)_{L^2}
		\end{aligned}
		\end{align}
		for all  $g_1, g_2 \in L^2(\R^3)$  with compact support.
	\end{enumerate}
\end{thm}
\begin{rem} \label{rem:main}
\begin{enumerate} [label*=(\roman*)]
		\item The analogous result holds when the Stokes equations are replaced by the Poisson equation. Also for the Poisson equation, the result is new, see the discussion in Section \ref{sec:discussion}.
  For the sake of conciseness, we do not state the result in a separate theorem but only point out the necessary adaptations: Instead of the Stokes equations \eqref{eq: CM}, \eqref{eq: CM_hom} and \eqref{eq: CM_lambda_hom} we consider Poisson equations and the quantities $V_i$ become scalars as well as $u_m, u, h, j$, etc.
  Moreover, reflecting that the capacity of a ball of radius $R$ is $4 \pi R$, one should replace \eqref{eq:R_m.Stokes} by
	\begin{align} \label{eq:R_m.Poisson}
		R_m = \frac 1 {4 \pi m}.
	\end{align}

    \item For the proof, we will show the following statement that implies (ii), see Theorem \ref{pro:tilde.u_m} and Propositions \ref{pro:xi_m}--\ref{pro:tau_m}: 
    Denoting $\Omega$ the underlying probability space, let $\sigma_m \in L^2(\Omega;L^2_{loc}(\R^3))$ be the i.i.d.\ random fields (cf. \eqref{def:tau_m}) given by
    \begin{align} \label{def.sigma}
        \sigma_i = m^{-1/2} \left(A(\rho u - j) - (A(u(X_i) - V_i) \delta_{X_i})\right)
     \end{align}
    and $\tau_m := \sum_{i=1}^m \sigma_i $. Then, for all $\beta < 1/2$
    \begin{align} \label{patwise.convergence}
        m^{1/\beta} \|m^{1/2}(u_m - u)  - \tau_m\|_{L^2(K)} \to 0 \quad \text{ in probability}.
    \end{align} 
     The assertion then follows from a standard CLT upon computing the covariance of $\sigma_1$.
     In classical stochastic homogenization of elliptic PDEs with oscillating coefficients, the analogue of the convergence \eqref{patwise.convergence} is known as \emph{pathwise structure of fluctuations} (see e.g. \cite{DuerinckxGloriaOtto20}). A heuristic explanation for this pathwise structure will be given in Section \ref{sec:Heuristics}.

	\item \label{it:nature.fluctuations} Formally, we can write $\xi = A \zeta$, where $\zeta$ accounts for the fluctuations of the drag force $j - \rho u$. The appearance of the second term on the right-hand side of \eqref{eq:covaranaice} is classical for the fluctuations in $m$-particle systems, see e.g. \cite{BraunHepp1977}, and is supposed to disappear if we modeled the particles by a Poisson Point Process instead.
In particular,  one can expect in this case, at least formally, $\xi = A \zeta$ with
\begin{align}
	\zeta = \left(\int (v-u)\otimes (v-u) f(\cdot, \de v)\right)^{\frac 1 2 } W, 
\end{align}
where $W$ is space white noise.	
This (and also the form of $\sigma_i$ in \eqref{def.sigma}) means that the fluctuations are solely caused by the fluctuations of the effective particle drag forces $V_i - u(X_i)$ due to the fluctuations of the positions and velocities of the particles. No other information on the Dirichlet boundary conditions is retained.
The fluctuations of the drag force are then transferred to fluctuations of the fluid velocity $u$ via the long range solution operator $A$ of the Brinkman equations.

\item
The rate of convergence in part (i) of Theorem \ref{th:main} is optimal in view of part (ii). More precisely,  since $\xi \neq 0$ in general, part (i) cannot hold for $\beta = 1/2$.
By interpolating the estimate in part (i) with the energy bound, one obtains a convergence in $H^s_\loc$ for any $s < 1$ with rate $m^{-\beta+s/2}$ for any $\beta < 1/2$, though. This might not be optimal, though. Indeed, we will show that the fluctuations $m^{1/2}(u_m-u)$ are bounded in $H^s_\loc$, $s < 1/2$ (cf. Proposition \ref{pro:tau_m}).
\end{enumerate}
\end{rem}

\subsubsection*{Possible generalizations}

We briefly comment on three aspects of possible generalizations and improvements of our main result. We address: 1) random radii of the particles; 2) more general  distributions of particle positions, 3) space dimensions different from $d = 3$; 4) different notions of probabilistic convergence in  part (i) of the main theorem.

\begin{enumerate}

\item Indeed, it is not difficult to extend the above result to the case where the radii of the particles are also random.
More precisely, assume that the radius of each particle is $R^m_i = r_i R_m$ with $R_m$ as in \eqref{eq:R_m.Stokes}, respectively. Assume that the radii $r_i$ are independent bounded random variables, also independent of the positions, with expectation $\E r =  1$.
Then, the assertions of Theorem \ref{th:main} still hold with an additional factor $\E r^2$ in front of the first term on the right-hand side of the covariance. In order not to further burden the presentation, we restrict our attention to the case of identical radii.

\item For the sake of simplicity, we restrict our analysis to $m$ i.i.d. particles.
As mentioned in Remark \ref{rem:main} (iii), we expect the result to extend to (inhomogeneous) Poisson Point Processes. Moreover, we expect similar results for sufficiently  mixing processes. For instance, assume that the $m$ particles are identically distibuted with $V_i = 0$ for all $1 \leq i \leq m$, i.e. $f = \rho \otimes \delta_0$ and let $\rho_2$ denote the $2$-particle correlation function, i.e., 
$\E (g(X_1,X_2)) = \int_{\R^3 \times \R^3} g(x_1,x_2) \dd \rho_2(x_1,x_2)$. Assume that 
the process is mixing in the sense of
\begin{align} \label{ass.correlations}
    \left|\rho_2\left(\frac{x_1}{m^{1/3}}, \frac{x_2}{m^{1/3}}\right) - \rho\left(\frac{x_1}{m^{1/3}}\right) \rho\left(\frac{x_2}{m^{1/3}}\right)  \right|\leq \left( 1+ \frac{|x_1-x_2|}{m^{1/3}}\right)^{-\beta} 
\end{align}
for some $\beta > 3$. Then,  $\tau_m = \sum_{i=1}^m \sigma_i$ with $\sigma_i$ as in \eqref{def.sigma} still has a bounded variance which seems necessary for the fluctuations to be of order $m^{-1/2}$. We point out, that the condition $\beta >3$ corresponds to the one under which the fluctuations have been shown to obey the central limit scaling in \cite{DuerinckxFischerGloria22} in the case of oscillating coefficients.
However, the probabilistic estimates in Section \ref{sec:proof.stochastic} involve expressions with up to $5$ different particles. Hence, more assumptions on the particle correlations are likely to be necessary when the particles are not independently distributed.

\item Regarding the space dimension, our analysis is restricted to the physically most relevant three-dimensional case. Applying the same techniques in dimension $d=2$ seems possible with additional technicalities due the usual issues regarding the capacity of a set in $d=2$.

We emphasize though that, for $d \geq 4$, we do not expect Theorem \ref{th:main} to continue to hold without structural changes. More precisely, we expect that in higher dimensions, the fluctuations occur at a higher rate (than $m^{-1/2}$). Moreover,  at leading order, we expect local effects to dominate rather than the long range fluctuations caused by the the fluctuations of the drag force in $d=3$ (cf. Remark \ref{rem:main} \ref{it:nature.fluctuations}).
The reason for this is that the volume occupied by the particles becomes too big.
Indeed, \rh{in order for the homogenized equation \eqref{eq: CM_hom} to remain the Brinkman equations,} the critical scaling  of the radius of $m$ spherical particles in dimension $d \geq 3$ is $R_m \sim m^{-1/(d-2)}$. The results cited above ensure that under this scaling we still have $u_m \rightharpoonup u$ weakly in $\dot H^1(\R^d)$.
However, in the case when the particle velocities are all zero, i.e. $f=\rho \otimes \delta_0$, we obtain as a trivial upper bound for the rate of convergence in $L^p_{\loc}$
\begin{align}
	\|u_m - u\|_{L^p_\loc(\R^3)} \geq \|u_m - u\|_{L^p(\cup_{i=1}^m B_i)} = \|u\|_{L^p(\cup_{i=1}^m B_i)} \sim \left(\L^d\left(\bigcup_{i=1}^m B_i\right)\right)^{\frac 1 p} \sim m^{-\frac{2}{p(d-2)}}.
\end{align}
This shows that Theorem \ref{th:main} cannot hold in this form for $d \geq 5$. Moreover, in dimension $d=4$, this error is of critical order, which suggests that the analysis of the fluctuations is much more delicate. 
One might expect, though, that the result remains true in $d \in \{4,5\}$ by changing the space $L^2_{\loc}$ to $L^p_{\loc}$ for $p$ sufficiently small such that $\|\1_{\cup_{i=1}^m B_i}\}\|_{L^p} \ll  m^{-1/2}$. However, inspecting more carefully that the effect of the Dirichlet condition at the particles decays like the inverse distance and that the  typical particle distance is $m^{-1/d}$, reveals that (with overwhelming probability)
\begin{align}
    \|u_m - u\|_{L^p_\loc(\R^3)} \gtrsim m^{-\frac 2 d}
\end{align}
for all $p < \tfrac{d}{d-2}$ (such that the fundamental solution of the Stokes equations is in $L^p_{\loc}$).

We do not only believe that the scaling of the fluctuations change but also there nature. Indeed, the  the long range fluctuations caused by the the fluctuations of the drag force in $d=3$ (cf. \eqref{def.sigma}) is not adapted to locally correct the failure of the Dirichlet boundary condition at the particles. Roughly speaking, the fluctuations in $d=3$ at a given point is to leading order a collective long range effect due to the fluctuations of \emph{all} particle positions and velocities. In $d \geq 5$, however, we expect the fluctuation to leading order to be a short range effect due to the fluctuation of the \emph{nearest} particle position and its velocity. For $d=4$, we expect both effects to be of the same order.


\item 

Instead of convergence in probability, one could aim for convergence in $L^p$. Following the proof of the theorem reveals that we actually prove
\begin{align}
	\E_m[\1_{\mathcal O_{m,0,5}} \|u_m - u\|^2_{L^2_\loc}] \leq C m^{-1}.
\end{align}
This implies $\E_m[\|u_m - u\|_{L^2_\loc}] \leq C m^{-1/6}$
by Lemma \ref{lem:prob.distances}, provided an a priori bound 
$\E_m[\|u_m - u\|^2_{L^2_\loc}] \leq C$. Such a bound has been obtained in \cite{CarrapatosoHillairet20}. Although different particle distributions are considerd in \cite{CarrapatosoHillairet20}, one readily checks that \cite[Lemma 3.4]{CarrapatosoHillairet20} also implies such an a priori estimate in our setting.
Again, the power $m^{-1/6}$ is presumably not optimal and one could aim for an estimate $\E_m[\|u_m - u\|^2_{L^2_\loc}] \leq C m^{-1}$. 
Following our present approach, one would need to adapt the approximation that we use for $u_m$ in the set $ \mathcal O_{m,0,5}$. The adaptation needs to take into account in a more precise way the geometry of the particle configuration and one could take inspiration from the proof of \cite[Lemma 3.4]{CarrapatosoHillairet20}. However, it seems unavoidable that this approach would drastically increase the technical part of our proof.

\end{enumerate}
\subsubsection*{Comments on assumption \ref{ass:energy}--\ref{ass:j}}

 The second moment bound in the first assumption, \ref{ass:energy}, is very natural. It ensures that the solution $u_m$ is bounded in $L^2(\Omega;\dot H^1(\R^3))$, where $\Omega$ denotes the probability space. Moreover, the covariance of the fluctuations provided in Theorem \ref{th:main} involve this second moment.

The regularity assumptions on $\rho$ and $j$, \ref{ass:rho}--\ref{ass:j}, are of more technical nature: they ensure that both $j$ and $\rho u$, which appear in the Brinkman equations \eqref{eq: CM_hom}, lie in $\dot H^1(\R^3) \cap \dot H^{-1}(\R^3)$. The $\dot H^{-1}$ property  will be very useful to treat those terms as source terms of the Stokes equations. On the other hand, the $H^1$-regularity allows us to quantify the differences of those terms to some discrete and averaged versions involved in the setup of appropriate approximations for $u_m$ that we detail in Section \ref{sec:Heuristics}.

\subsection{Discussion of related results}

\subsubsection*{Previous results on the derivation of the Brinkman equations} \label{sec:discussion}

As indicated at the beginning of this introduction, there is a huge literature on the derivation of the Brinkman equations and corresponding results for the Poisson equation where one could mention 
for instance \cite{MarchenkoKhruslov74,CioranescuMurat82a,papvar.tinyholes, Ozawa83, DalMasoGarroni.punctured, GiuntiHoferVelazquez18}. 
For a more complete list and discussion of this literature, we refer the reader to \cite{GiuntiHoferVelazquez18,GiuntiHoefer19}.

In \cite{GiuntiHoefer19,CarrapatosoHillairet20}, the Brinkman equations have been derived under very mild assumptions on the particle configurations. In \cite{GiuntiHoefer19}, the authors considered zero particle veolcities. The particle positions can be distributed to rather general stationary processes, and the radii are i.i.d. with only a $(1+ \beta)$ moment bound. This allows for many clusters of overlapping particles. A corresponding result for the Poisson equation has been obtained in \cite{GiuntiHoferVelazquez18}.

On the other hand, in \cite{CarrapatosoHillairet20}, the particle radii are identical but their velocities are not necessarily zero. The authors consider more general particle distributions than i.i.d. configurations. The Brinkman equations are derived in this setting under assumptions including a $5$th moment bound of the velocities.
The result in \cite{CarrapatosoHillairet20} comes with an estimate of the convergence rate $u_m \to u$ in $L^2_\loc$. However, this does not allow to deduce convergence faster than $m^{-\beta}$ with $\beta < 1/95$.

\subsubsection*{Results about explicit approximations for $u_m$}

A widespread approach  to homogenization of the Poisson and Stokes equations in perforated domains with homogeneous Dirichlet boundary conditions  is the so-called method of oscillating testfunctions which is used for instance in \cite{CioranescuMurat82a,Allaire90a}. An oscillating testfunction $w_m$ is constructed in such a way that it vanishes in the particles and converges to $1$ weakly in $H^1_\loc$. This function $w_m$ carries the information of the capacity (or resistance) of the particles.
 A natural question is then, how well $w_m u$ approximates $u_m$. Since the function $w_m$ is usually constructed explicitly, 
 this allows for an explicit approximation for $u_m$.
 In \cite{KacimiMurat89, Allaire90a} it is shown that for periodic configurations $\|u_m - w_m u\|_{\dot H^1} \leq C m^{-1/3}$. This error is of the order of the particle distance and thus  the optimal error that one can expect due to the discretization. Similar results have been obtained in \cite{Giunti20} for the random configurations studied in \cite{GiuntiHoferVelazquez18}, with a larger error due to particle clusters. 

In the recent papers \cite{Feppon22, FepponWenjia21}, higher order approximations for the Poisson and the Stokes equations in periodically perforated domains are analyzed.
 
In the present paper, we do not work with oscillating test functions. However, we derive equally explicit approximations 
for $u_m$ which we will denote by $\tilde u_m$ (see Section \ref{sec:Heuristics}).
As we will show in Theorem~\ref{pro:tilde.u_m} , we have $\|u_m - \tilde u_m \|_{\dot H^1} \leq C m^{-\beta}$ for all $\beta < 1$. This error is much smaller than the one obtained in \cite{KacimiMurat89, Allaire90a}. The reason for that is twofold. First, we take into account the leading order discretization error in terms of fluctuations. Second, we benefit from the randomness which reduces the higher order dicretization errors on average.
We believe that Theorem \ref{pro:tilde.u_m} could be of independent interest. In particular concerning the rigorous derivation of the Vlasov-Stokes equations \eqref{eq:Vlasov.Stokes}, such explicit accurate approximations of $u_m$ in good norms seem essential. Indeed, for the related derivation of the transport-Stokes system for inertialess suspensions in \cite{Hofer18MeanField}, corresponding approximations have been crucial.

\subsubsection*{Related results concerning fluctuations and preliminary comments on our proof}

In the classical theory of stochastic homogenization of elliptic equations with oscillating coefficients, the study of fluctuations has been a very active research field in recent years. Of the vast literature, one could mention for example   \cite{ArmstrongKuusiMourrat17, DuerinckxGloriaOtto20, DuerinckxGloriaOtto20a, DuerinckxFischerGloria22}.

Regarding the homogenization in perforated domains, the literature is much more sparse. In the recent paper \cite{DuerinckxGloria21}, the authors were able to adapt some of the techniques of quantitative stochastic homogenization of elliptic equations with oscillating coefficients to the Stokes equations in perforated domains with sedimentation boundary conditions which are different from the ones considered here.

Related results to Theorem \ref{th:main} have been obtained in \cite{FigariOrlandiTeta85} for the Poisson equation and
in \cite{Rubinstein1986} for the Stokes equations. 
However, in these papers, the authors were only able to treat the Poisson and the Stokes equations corresponding to \eqref{eq: CM} with an additional large massive term $\lambda u_m$: they obtained a result corresponding to Theorem \ref{th:main} provided that $\lambda$ is sufficiently large (depending on $\rho$).

The approach in \cite{FigariOrlandiTeta85,Rubinstein1986} follows the approximation of the solution $u_m$ by the so-called method of reflections.
The idea behind this method is to express the solution operator of the problem in the perforated domain in terms of the solutions operators when only
one of the particles is present. More precisely, let $v_0$ be the solution of the problem in the whole space without any particles. Then, define 
$v_1 = v_0 + \sum_i v_{1,i}$ in such a way that $v_0 + v_{1,i}$ solves the problem if $i$ was the only particle. Since $v_{1,i}$ induces an error in $B_j$ for $j \neq i$, one adds further functions $v_{2,i}$, this time starting from $v_1$.
Iterating this procedure yields a sequence $v_k$. In general, $v_k$ is not convergent. With the additional massive term though,
one can show that the method of reflections does converge, provided that $\lambda$ is sufficiently large.

In \cite{HoferVelazquez18}, the first author and 
Vel\'azquez showed how the method of reflections can be modified to ensure convergence without a massive term and how this modified method can be used to obtain convergence results for the homogenization of the Poisson and Stokes equations.
In order to study the fluctuations, a high accuracy of the approximation of $u_m$ is needed. This would make it necessary to analyze many of the terms arising from the modified method of reflections which we were allowed to disregard for the qualitative convergence result of $u_m$ in \cite{HoferVelazquez18}.
It seems very hard to control sufficiently well these additional terms which either do not arise or are of higher order for the (unmodified) method of reflections used in \cite{FigariOrlandiTeta85,Rubinstein1986}.

Thus, in the present paper, we do not use the method of reflections but follow an alternative approach to obtain an approximation for $u_m$.
Again, we approximate $u_m$ by $\tilde u_m = w_0 + \sum_i w_{i}$, where $w_{i}$ solves the homogeneous Stokes equations outside of $\overline B_i$. However, we do not take $w_{i}$ as in the method of reflections, where it is expressed in terms of $w_0$.
Instead $w_{i}$ will depend on $u$,  exploiting that we already know that $u_m$ converges to $u$. In contrast to the approximation obtained from the method of reflections, we will be able to choose $w_{i}$ in such a way that the approximation  $\tilde u_m = w_0 + \sum_i w_{i}$ is sufficient to capture the fluctuations.

A related approach has recently been used in a parallel work by G\'erard-Varet in \cite{Gerard-Varet19} to give a very short proof of the homogenization 
result $u_m \rightharpoonup u$ weakly in $\dot H^1$ under rather mild assumptions on the positions of the particles. 
However, since we study the fluctuations in this paper, we need a more refined approximation than the one used in \cite{Gerard-Varet19}.
More precisely, to leading order, the function $w_{i}$ will  only depend on $V_i$ and the value of $u$ at $B_i$.
However, $w_i$ will also include a lower-order term which is still relevant for the fluctuations. As we will see, this lower-order term will  depend in some way on the fluctuations of the positions of all the other particles.

\subsection{Organization of the paper}

The rest of the paper is devoted to the proof of the main result, Theorem \ref{th:main}. 

In Section \ref{sec:Heuristics}, we give a precise definition of the  approximation $\tilde u_m = w_0 + \sum_i w_{i}$, outlined in the paragraph above,
as well as a heuristic explanation for this choice.

In Section \ref{sec:proof.main}, we state three key estimates regarding this approximation and show how the proof of Theorem \ref{th:main} follows from these estimates.

The proof of these key estimates contains a purely analytic part as well as a stochastic part which are given in Sections \ref{sec:proof.analytic}
and \ref{sec:proof.stochastic}, respectively.

\section{The approximation for the microscopic solution \texorpdfstring{$u_m$}{um}} \label{sec:Heuristics}

\subsection{Notation} \label{sec:MoreNotation}

We introduce the following notation that is used throughout the paper.

We denote by $G \colon \dot H^{-1}(\R^3) \to \dot H^1(\R^3)$ the solution operator for the Stokes equations.
This operator is explicitly given as a convolution operator with kernel $g$, the fundamental solution to the Stokes equations, i.e.,
	\begin{align} \label{eq:fund.sol}
		g(x) = \frac{1}{8 \pi}\left( \frac{\mathrm{Id}}{|x|} + \frac{x \otimes x}{|x|^3} \right).
	\end{align}

We recall from Theorem \ref{th:main} that $A \colon \dot H^{-1}(\R^3) \to \dot H^1(\R^3)$ is the solution operator for the limit problem \eqref{eq: CM_lambda_hom}.
We observe the identities
\begin{align} \label{eq:relation.A.G}
	(1 + G\rho) A = G, && A (1+\rho G) = G, && A = G - A \rho G.
\end{align}
We remark that multiplication by $\rho$ maps from $\dot H^1(\R^3)$ to $H^1(\R^3) \cap \dot H^{-1}(\R^3)$. Indeed, this follows from $\rho \in W^{1,\infty}(\R^3)$ with compact support and the fact that
$\dot H^1(\R^3) \subset L^6(\R^3)$ which implies $L^{6/5}(\R^3) \subset \dot H^{-1} (\R^3)$.
%
Furthermore, observe that $A$ and $G$ are bounded operators from $L^2(\R^3) \cap H^{-1}(\R^3)$ to $C^{0,\alpha}(\R^3)$, $\alpha \leq 1/2$, and from \(H^1(\R^3)\cap H^{-1}(\R^3)\) to \(W^{1,\infty}(\R^3)\). In particular,
$A \rho$ and $G \rho$ are  bounded operators from $L^2(\supp \rho)$ (and in particular from $\dot H^1(\R^3)$) to $L^\infty(\R^3)$ and from $\dot H^1(\R^3)$ to $W^{1,\infty}(\R^3)$. 

We denote $G^{-1}= -\Delta$. Then we have $G G^{-1} = G^{-1} G = P_\sigma$, where $P_\sigma$ is the projection to the divergence free functions. In fact,  we will use $G^{-1}$ in the expression $A G^{-1}$ only. We observe that $A = A P_\sigma$ and thus
\begin{align}
	A G^{-1} G = A.
\end{align}

We denote by \(B^m(x) = B_{R_m}(x)\) and the normalized Hausdorff measure on the sphere $\partial B^m(x)$ by
\begin{align} \label{def:dirac.sphere}
	\delta^m_x:=  \frac{ \H^2|_{\partial B^m(x)}}{\H^2(\partial B^m(x))},
\end{align}
and write \(\delta^m_i:=\delta^m_{X_i}\).

Moreover, we denote for any function $\varphi \in L^1(B^m(x))$ the average on $B^m(x)$ by $(\varphi)_x$, i.e.
\begin{align} \label{def:average.ball}
	(\varphi)_x := \fint_{B^m(x)} \varphi(y) \dd y :=  \frac{1}{|B^m(x)|} \int_{B^m(x)} \varphi(y) \dd y,
\end{align}
and we abbreviate $(\varphi)_i := (\varphi)_{X_i}$.

We will need a cut-off version of the fundamental solution. To this end, let \(\eta\in C^{\infty}_c(B_3(0))\) with \(\1_{B_2(0)} \leq \eta \leq \1_{B_3 (0)}\) and $\eta_m (x) := \eta(x/R_m)$. 
Now consider \(\tilde{g}_m=(1-\eta_m)g\). We need an additional term in order to make $\tilde{g}^m$ divergence free. This is obtained through the classical Bogovski operator 
(see e.g. \cite[Theorem 3.1]{Galdi11})
which provides the existence of a sequence
 $\psi_m \in C_c^\infty(B_{3 R_m} \setminus B_{2 R_m})$  such that
$\dv \psi_m = \dv (\eta_m g)$ and 
\begin{align} \label{eq:Bogovski.estimate}
	\|\nabla^k \psi_m\|_{L^p(\R^3)} \leq C(p,k) \| \nabla^{k-1} \dv (\eta_m g)\|_{L^p(\R^3)}
\end{align}
for all $1 < p < \infty$ and all $k \geq 1$.
By scaling considerations, the constant $C$ is independent of $m$.
Then, we define $G^m$ as the convolution operator with kernel 
\begin{align} \label{eq:g^m.Stokes}
	g^m = (1-\eta_m) g + \psi_m.
\end{align}

\subsection{Approximation of \texorpdfstring{$u_m$}{um} using monopoles induced by \texorpdfstring{$u$}{u}}

To find a good approximation for $u_m$, we observe that $u_m$ satisfies
\begin{align} \label{eq:u_m.whole.space}
	- \Delta u_m + \nabla p = h \1_{\Omega_m} + \sum_i h_i, \quad \text{in } \R^3
\end{align}
for some functions $h_i \in \dot H^{-1}(\R^3)$, each supported in $\overline B_i$, which are the force distributions induced in the particles due to the Dirichlet boundary conditions.

We begin by observing that for most of the configurations of particles, the particles are sufficiently separated
which allows us to  determine good approximations for $h_i$ by ignoring its direct interaction with another particle. As we will see, our approximation for $h_i$ will only incorporate the effect of the other particles 
through the limit $u$.

To be more precise, let $0 < \nu < 1/3$. Then, by Lemma \ref{lem:prob.distances}, we know that, for most of the particles, $B_{m^{\nu} R_m }(X_i)$ only contains the particle $B_i$.
In this case, $h_i$ is uniquely determined by the problem
\begin{align} \label{eq:v_i}
	\left\{\begin{array}{rl}
		-\Delta v_{i} + \nabla p   =  h \quad &\text{in } B_{m^{\nu} R_m }(X_i) \setminus \overline{B_i}, \\
		v_{i}  = V_i  \quad &\text{in }  \overline B_i , \\
		v_i = u_m  \quad &\text{on }  \partial B_{m^{\nu} R_m }(X_i).
	\end{array} \right.
\end{align}
We simplify this problem to derive an approximation for $h_i$. First, we drop the right-hand side $h$ in \eqref{eq:v_i}. Its contribution is expected to be negligible, since the volume of $B_{m^{\nu} R_m }(X_i) \setminus \overline{B_i}$ is small compared to the difference of the boundary data at $\partial B_i$ and $\partial  B_{m^{\nu} R_m }(X_i)$ which is typically of order $1$.
Next, we know that typically $\partial B_{m^{\nu} R_m }(X_i)$ is very far from any particle. Since $u_m \wto u$ in $\dot H^1(\R^3)$, we therefore replace \eqref{eq:v_i} by
\begin{align}\label{eq:v_i.app}
	\left\{\begin{array}{rl}
		-\Delta v_{i} + \nabla p =  \R^3 \setminus \overline{B_i}, \\
		v_{i}  = V_i \quad &\text{in }  \overline B_i , \\
		v_i(x) \to (u)_{i} \quad &\text{as }  |x - X_i| \to \infty.
	\end{array} \right.
\end{align}
Here, we could also have chosen $u(X_i)$ instead of $(u)_i$. The precise choice that we make will turn out to be convenient later.
By our choice of $R_m$ in  \eqref{eq:R_m.Stokes}, the explicit solution of \eqref{eq:v_i.app}  is given by $v_i$ which solves $- \Delta v_i + \nabla p = h_i$ in $\R^3$ with
\begin{align}
	h_i = \frac {V_i - (u)_i} m  \delta^m_i.
\end{align}
Therefore, resorting to \eqref{eq:u_m.whole.space}, we  are led to approximate $u_m$ by
\begin{align} \label{eq:def.tilde.u_m.no.fluc}
	\tilde u_m := G\left[h - \frac 1 m \sum_{i=1}^m {((u)_i - V_i)} \delta^m_i\right].
\end{align}
We emphasize that for this approximation it is not important to know the function $u$. We only used that $u_m \wto u$ in $\dot H^1(\R^3)$
which is always true for a subsequence by standard energy estimates.
On the contrary, we can now identify the limit $u$. Indeed, if we believe that $\tilde u_m$ approximates $u_m$ sufficiently well,
\begin{align} \label{eq:u.a.posteriori}
	u \leftharpoonup   u_m \approx \tilde u_m = G\left[h - \frac 1 m \sum_{i=1}^m {((u)_i - V_i)} \delta^m_i\right] \wto G[h + j- \rho u]
\end{align}
which shows that $u$ indeed solves \eqref{eq: CM_hom}.

This approximation $\tilde u_m$  cannot fully capture the fluctuations, though. 
In the next subsection we thus show how to refine this approximation.

We end this subsection by comparing this approximation to the one  used in \cite{FigariOrlandiTeta85,Rubinstein1986} through the method of reflections. The first order approximation of the method of reflections is given by $\tilde u_m$ as defined in \eqref{eq:def.tilde.u_m.no.fluc} but with $G h$ instead of $u$ on the right-hand side. Since this is a much cruder approximation,
one needs to iterate the approximation scheme. This only yields a convergent series in \cite{FigariOrlandiTeta85,Rubinstein1986} due to the additional large massive term. On the other hand, this series then approximates $u_m$ sufficiently well without the refinement that we introduce in the next
subsection.

\subsection{Refined approximation to capture the fluctuations}

We make the ansatz that, macroscopically,
	\begin{align} \label{eq:u_m.xi_m}
		u_m = u + m^{-\frac{1}{2}} \xi_m + o(m^{-\frac{1}{2}}),
	\end{align}
where $\xi_m$ is a random function which needs to be determined.
We assume that the fluctuations $\xi_m$ are in some sense macroscopic, just as $u$, such that we can follow the same approximation scheme as in the previous subsection.

More precisely, we adjust the Dirichlet problem \eqref{eq:v_i.app} by adding  $m^{-\frac{1}{2}} (\xi_m)_i$ on the right-hand side of the third line.
This leads to the definition
\begin{align} \label{def:tilde.u_m}
	\tilde{u}_m := G\left[ h - \frac 1 m \sum_{i=1}^{m} {(u - V_i +m^{-\frac{1}{2}} \xi_m)_i} \delta^m_i \right].
\end{align}

We have not defined $\xi_m$ yet. To make a good choice for $\xi_m$, the idea is to use a similar argument as in \eqref{eq:u.a.posteriori}
but only to take the limit $m \to \infty$ in terms which are of lower order.
More precisely, we observe, again taking for granted that $\tilde u_m$ approximates $u_m$ sufficiently well
and using \(u = G(h + j - \rho u) \),
\begin{equation} \label{eq:xi_m.first}
\begin{aligned}
	u + m^{-1/2} \xi_m & \approx u_m \approx \tilde u_m = G\left[ h - \frac 1 m \sum_{i=1}^{m} (u - V_i +m^{-\frac{1}{2}} \xi_m)_i \delta^m_i \right] \\
	&= u + G\left[ \rho u - j - \frac 1 m \sum_{i=1}^{m} ((u)_i - V_i) \delta^m_i \right] - G\left[ \frac 1 m \sum_{i=1}^{m} (m^{-\frac{1}{2}}  \xi_m)_i \delta^m_i\right]. \\
\end{aligned}
\end{equation}

We expect
\begin{align}\label{eq:fluc.xi_m}
	G\left[\sum_{j\neq i} \frac{m^{-\frac{1}{2}}(\xi_m)_j}{m} \delta^m_j \right] = G(\rho m^{-\frac{1}{2}}\xi_m) + O(m^{-1}).
\end{align}

Inserting this into \eqref{eq:xi_m.first}, leads to
\begin{align}\label{eq:xi_m.wrong}
	m^{-1/2} \xi_m +  G(\rho m^{-\frac{1}{2}}\xi_m)  \approx G\left[ \rho u - j -  \frac 1 m \sum_{i=1}^{m} ((u)_i - V_i) \delta^m_i \right].
\end{align}

This equation could be used as a definition of $\xi_m$.
Although this turns out to be a good approximation on the level of equation \eqref{eq:u_m.xi_m}, we will now argue that this is not the case for the definition of $\tilde u_m$ in \eqref{def:tilde.u_m}.
Indeed,  the right-hand side of \eqref{eq:xi_m.wrong} is equal to $(u)_i - V_i$ in $B_i$ to leading order.
Hence, $(m^{-1/2}\xi_m)_i$ would be of the same order which would yield a contribution to $\tilde u_m$ through $\xi_m$ of order $1$ instead of order $m^{-1/2}$.

Therefore, we need to be more careful and go back to microscopic considerations:
Since $u_m = V_i$ in $B_i$ and $\tilde u_m \approx u_m$, we want to define $\xi_m$ in such a way that $\tilde{u}_m \approx V_i$ in $B_i$.
Thus we want to compute $\tilde u_m$ in $B_i$ in order to find a good definition of $\xi_m$. Since we expect $\tilde u_m = \tilde u_m (X_i) + O(m^{-1})$ in
$B_i$  (at least on average), we only compute $\tilde u_m(X_i)$, and by the same reasoning, we replace any average $(\xi_m)_i$ by $\xi_m(X_i)$ at will.
Then, we find, using again $u = G(h + j - \rho u)$,
\begin{equation} \label{eq:tilde.u_m(w_i)}
\begin{split}
	\tilde{u}_m(X_i) & \approx u(X_i) + (G (\rho u - j))(X_i) - u(X_i) + V_i - m^{-\frac{1}{2}}\xi_m(X_i) \\
	& \qquad - G\left[\frac 1 m \sum_{j\neq i} {(u - V_j +m^{-\frac{1}{2}} \xi_m)_j} \delta^m_j \right](X_i) \\
	& =V_i  -m^{-\frac{1}{2}} \xi_m(X_i) + G\left[\rho u - j -  \frac 1 m \sum_{j\neq i} ((u)_j - V_j  +m^{-\frac{1}{2}} \xi_m)_j) \delta^m_j\right](X_i) . 
\end{split}
\end{equation}
Requiring $\tilde{u}_m(X_i) = V_i$ yields
\begin{equation} \label{eq:xi_m.0}
m^{-\frac{1}{2}} \xi_m(X_i) + G\left[ \frac 1 m \sum_{j\neq i} {m^{-\frac{1}{2}} (\xi_m)_j}\delta^m_j \right](X_i) =  G\left[\rho u - j - \frac 1 m \sum_{j\neq i} ((u)_j - V_j) \delta^m_j\right](X_i).
\end{equation}
In order to define $\xi_m$ from this equation, we want the sum on the right-hand side to include $i$ such that the function is the same for every $i$. 
To this end, we notice that by Lemma \ref{lem:prob.distances}, with high probability, we have for all $i$ and all $W \in \R^3$
	\begin{align} \label{eq:prop.G^m_lambda}
		G^m \delta^m_i W = 0 \quad \text{in } B_i, \qquad G \delta^m_j W = G^m \delta^m_j W \quad \text{in } B_i  \quad \text{for all } j \neq i,
	\end{align}
where $G^m$ is the operator introduced at the end of Section \ref{sec:MoreNotation}.
Hence, we replace the right-hand side of \eqref{eq:xi_m.0} by
	\begin{align} \label{def:Theta}
		m^{-\frac{1}{2}} \Theta_m := G (\rho u - j ) -  \frac{1}{m} \sum_{i=1}^{m} {G}^m \left(((u)_i - V_i)\delta^m_i\right).
	\end{align}
We expect $\Theta_m \sim 1$ since the right-hand side of \eqref{def:Theta} represents the fluctuations of the discrete approximation of $G (\rho u - j)$.
As before, we replace the sum on the left-hand side of \eqref{eq:xi_m.0} by $\rho \xi_m$.
Combining these approximations leads to
\begin{align} \label{eq:relation.xi.Theta}
m^{-\frac{1}{2}}(1 + G \rho)  \xi_m = m^{-\frac{1}{2}} \Theta_m.
\end{align}
In view of \eqref{eq:relation.A.G}, it holds \((1+G\rho)AG^{-1} = P_\sigma\). Since, $\Theta_m$ is divergence free, \eqref{eq:relation.xi.Theta} leads to  define \(\xi_m\) to be the solution of 
\begin{align} \label{def:xi}
	\xi_m = AG^{-1} \Theta_m.
\end{align}
Note that the only difference between this definition of $\xi_m$ and \eqref{eq:xi_m.wrong} is the replacement of $G$ by $G^m$.
As mentioned above, we expect that, on a macroscopic scale, the operators $G$ and $G^m$ are almost the same (we will make this argument rigorous in Lemma \ref{lem:G-1Gm}). Therefore, in equation \eqref{eq:u_m.xi_m}, we expect, that it does not play a role (in $L^2_\loc(\R^3)$)
whether we take $G$ or $G^m$.
Consequently, as an  approximation for $\xi_m$, we introduce
\begin{align} \label{def:tau_m}
	\tau_m &:= AG^{-1} \tilde \Theta_m, \\
	m^{-1/2} \tilde \Theta_m &:= G (\rho u - j)  - \frac 1 m \sum_{i=1}^m G((u(X_i) - V_i) \delta_{X_i}).
\end{align}
This function bears the advantage that it is the sum of i.i.d. random variables.
Hence, it is straightforward to study the limit properties of $\tau_m[g] := (g,\tau_m)$. Notice that we both replaced the average \((u)_i\) by the value in the center of the ball \(u(X_i)\) and \(\delta^m_i\) by \(\delta_{X_i}\). Since \(u\in\dot{H}^1(\R^3)\), \(\tau_m\) is not defined for every realization of particles. However, as we will see, it is well-defined as an \(L^2\)-function on the probability space with values in \(L^2_{\loc}(\R^3)\).

\section{Proof of the main result} \label{sec:proof.main}

The first step of the proof is to rigorously justify the approximation of $u_m$ by $\tilde u_m$,
defined in \eqref{def:tilde.u_m} with $\xi_m$ and $\Theta_m$ as in \eqref{def:xi} and \eqref{def:Theta}.

\begin{thm} \label{pro:tilde.u_m}
	For all $\eps > 0$ and all $\beta<1$
	\begin{align}
		\lim_{m \to \infty} \P_m \left[m^{\beta} \|u_m - \tilde u_m\|_{\dot H^1(\R^3)} > \eps \right] \to 0.
	\end{align}
\end{thm}
The next step is to show that we actually have 
\begin{align}
	\tilde u_m = u + m^{-1/2} \xi_m + o(m^{-1/2})
\end{align}
which was the starting point of our heuristics, i.e. \(\xi_m\) indeed 
describes the fluctuations of \(\tilde{u}_m\) around \(u\). In contrast to 
Theorem \ref{pro:tilde.u_m}, we can only expect local \(L^2\)-estimates since not even \(u_m-u\) is small in the strong topology of  \(\dot H^1(\R^3)\).

\begin{prop}\label{pro:xi_m}
	For all \(\eps>0\), all bounded sets \(K'\subset \R^3\) and all \(\beta< 1\)
		\[\lim\limits_{m\to\infty} \P_m\left[m^{\beta}\|\tilde u_m - u - m^{-1/2} \xi_m\|_{L^2(K')}> \eps \right] \to 0.\]
	
\end{prop}

Combining Proposition \ref{pro:tilde.u_m} and \ref{pro:xi_m}, we observe that we only have to prove the statements of Theorem \ref{th:main} with $u_m - u$ replaced by $m^{-1/2}\xi_m$.  We postpone the proofs of Theorem \ref{pro:tilde.u_m} and Proposition \ref{pro:xi_m} to Section \ref{sec:proof.analytic}.

The next proposition shows that, instead of $\xi_m$,  we can actually consider $\tau_m$ introduced in the previous section.

\begin{prop} \label{pro:tau_m}
	For any bounded set \(K'\subset\R^3\) and every \(0\leq s < \frac{1}{2}\) there is a constant \(C_s(K')>0\) independent of \(m\) such that
	\begin{align}\label{eq:xi.L^2}
	\E_m[\|\xi_m\|^2_{{H}^s(K')}] \leq C_s(K').
	\end{align}
	Let $\tau_m$ be defined by \eqref{def:tau_m}. Then,
	\begin{align}\label{eq:xi-tau}
		\limsup_{m \to \infty} m^{1-2s}\E_m \left[\|\xi_m - \tau_m\|_{{H}^s(K')}^2\right] \leq C_s(K').
	\end{align}
%
\end{prop}

 We postpone the proof of Proposition \ref{pro:tau_m} to Section \ref{sec:Proof.tau_m}.

Note that for \(s=0\), these estimates include the case \(L^2(K')\) which we will use now in order to prove Theorem \ref{th:main}.
Indeed, Theorem \ref{th:main} is a direct consequence of the above results together with the classical Central Limit Theorem.
\begin{proof}[Proof of Theorem \ref{th:main}] 
	Due to the uniform bound on \(\E_m[\|\xi_m\|^2_{L^2(K)}]\) from Proposition \ref{pro:tau_m}, assertion \ref{it:L^2.Convergence} of the main theorem follows immediately from  Theorem \ref{pro:tilde.u_m} and Proposition \ref{pro:xi_m} since \(\dot{H}^1(\R^3)\) embeds into \(L^2_{loc}(\R^3)\).
	
	Since convergence in probability implies convergence in distribution, Theorem \ref{pro:tilde.u_m} and Propositions \ref{pro:xi_m}
	and \ref{pro:tau_m} imply that it suffices to prove assertion \ref{it:fluctuations} of Theorem \ref{th:main} with $\xi_m[g]$ replaced by
	$\tau_m[g]:= (g, \tau_m)_{L^2(\R^3)}$, i.e we need to prove that
		\[\tau_m[g] \to \xi[g]\]
	 in distribution for any \(g\in L^2(\R^3)\) with compact support. Since \(\tau_m[g]\) is a sum of independent random variables, this is a direct consequence of the Central Limit Theorem and the following computation for covariances: let \(g_1,g_2\in L^2(\R^3)\) with compact support, then
	 
	\begin{align}
		&\E_m\left[\tau_m[g_1] \tau_m[g_2]\right] \\
		&= m^{-1} \E_m \left[ \left(g_1, \sum_{i=1}^{m} A \left(\rho u - j  - (u(X_i) - V_i) \delta_{X_i} \right) \right)_{L^2(\R^3)}  \right.\\
		& \qquad \qquad \qquad	\left. \left(g_2, \sum_{j=1}^{m} A \left(\rho u - j  - (u(X_j) - V_j) \delta_{X_j} \right) \right)_{L^2(\R^3)} \right] \\
		&= \int_{\R^3\times\R^3} \left(g_1,A (\rho u - j - (u(x) - v) \delta_x)\right)_{L^2(\R^3)} \left(g_2,A (\rho u - j - (u(x) - v) \delta_x)\right)_{L^2(\R^3)} f(\de x,\de v)\\
		&= \int_{\R^3\times\R^3} \left(g_1,A ((u(x) - v) \delta_x)\right)_{L^2(\R^3)} \left(g_2,A ((u(x) - v) \delta_x)\right)_{L^2(\R^3)} f(\de x,\de v) \\
		& \qquad - (A g_1, \rho u - j)_{L^2(\R^3)} (A g_2, \rho u - j)_{L^2(\R^3)} \\
		&=  \int_{\R^3\times\R^3} ((u(x)-v) \cdot (A g_1)(x)) ((u(x)-v) \cdot (A g_2)(x)) f(\de x,\de v)\\
		& \qquad -  \left(\rho u  -j,A g_1\right)_{L^2(\R^3)}\left(\rho u -j,Ag_2\right)_{L^2(\R^3)}.
	\end{align}
	Here we used that \(A\delta_x\in L^2_{\loc}(\R^3)\) (see Lemma \ref{lem:delta-delta_m})  and that \(A\) is a symmetric operator on \(L^2(\R^3)\). 
	This finishes the proof.
\end{proof}

\section{Proof of Theorem \ref{pro:tilde.u_m} and Proposition \ref{pro:xi_m}} \label{sec:proof.analytic}

In this section, we will reduce the proof of Theorem \ref{pro:tilde.u_m} and Proposition \ref{pro:xi_m} to 
proving the following single probabilistic lemma. The proof of this lemma, which is given in Section \ref{sec:proof.variance.in.balls},
is the main technical part of this paper. It makes rigorous the heuristic equation \eqref{eq:fluc.xi_m}.

As we discussed in the heuristic arguments, we will exploit in the following that the probability of having very close particles is vanishing as stated in Lemma \ref{lem:prob.distances}.
In the notation of this lemma, we abbreviate
\begin{align} \label{def:W_m}
	 \Ocal_m = \Ocal_{m,0,5}.
\end{align}

\begin{lem} \label{lem:variance.in.balls}
	Let $\Lambda_m$,$\Gamma_m$, $\Xi_m$ and $\tilde \Xi_m$ be defined by
	\begin{align}
		\Lambda_m & := (G^m - G) \left( \frac 1 m \sum_i ((u)_i - V_i) \delta^m_i \right) \label{def:Lambda_m}, \\
		\Gamma_m &:= G^m\left[\sum_{i} \frac{(u)_i - V_i}{m} \delta^m_i\right] + G(\rho m^{-\frac{1}{2}} \xi_m)  \label{def:Gamma_m},\\
 		\Xi_m &:= G (\rho m^{-\frac{1}{2}} \xi_m) - G^m\left[\sum_{i} \frac{m^{-\frac{1}{2}} (\xi_m)_i}{m} \delta^m_i \right] \label{def:Xi_m}, \\
 		 	\widetilde \Xi_m &:= G (\rho m^{-\frac{1}{2}} \xi_m) - G\left[\sum_{i} \frac{m^{-\frac{1}{2}} (\xi_m)_i}{m} \delta^m_i \right] \label{def:Xi_m.tilde}.
	\end{align}
	Then,
	\begin{align}
		\limsup_{m\to \infty}  m^2 \E_m\left[\1_{\Ocal_m} \|\nabla \left(Gh + \Gamma_m  + \Xi_m\right)\|^2_{L^2 (\cup_i B_i)} \right] < \infty, \\
		\limsup_{m\to \infty}  m^4 \E_m\left[\1_{\Ocal_m}  \|\Xi_m\|^2_{L^2 (\cup_i B_i)} \right] < \infty, \\
		\limsup_{m\to \infty}   m^2 \E_m\left[\1_{\Ocal_m} \| \widetilde\Xi_m + \Lambda_m \|^2_{L^2_\loc(\R^3)}\right]  <\infty.
	\end{align}
\end{lem}

The proof of this lemma is the main technical work of the present paper. We postpone it to Section \ref{sec:proof.variance.in.balls}.

\begin{proof}[Proof of Proposition \ref{pro:xi_m}]
Recall the definition of $\tilde u_m$ from \eqref{def:tilde.u_m}.
	We compute using $u = G (h - \rho u + j)$ and $\xi_m = A G^{-1} \Theta_m = \Theta_m - G \rho \xi_m$  (cf. \eqref{eq:relation.A.G})  and 
	the definition of $\Theta_m$ from \eqref{def:Theta}
	\begin{align}
		\tilde u_m - u - m^{-1/2} \xi_m &= G\left(h - \frac 1 m \sum_i (u - V_i + m^{-1/2} \xi_m )_i \delta^m_i\right) - u - m^{-1/2} \xi_m\\
		&= G\left(\rho u - j - \frac 1 m \sum_i (u - V_i + m^{-1/2} \xi_m )_i \delta^m_i \right) - m^{-1/2} \Theta_m + m^{-1/2} G \rho \xi_m  \\
		&= m^{-1/2} G \left( \rho \xi_m - \frac 1 m \sum_i (\xi_m)_i \delta^m_i\right) + (G^m - G) \left( \frac 1 m \sum_i ((u)_i  -V_i )\delta^m_i \right)\\
		&= \widetilde \Xi_m + \Lambda_m.
	\end{align}
	Hence,
	\begin{align}
		&\P_m\left[m^{\beta}\|\tilde u_m - u - m^{-1/2} \xi_m\|_{L^2(K')}> \eps \right]\leq \P_m[\Ocal_m^c] + C\eps^{-2}m^{2\beta}\E_m\left[\1_{\Ocal_m} \| \widetilde\Xi_m + \Lambda_m \|^2_{L^2(K')}\right]
	\end{align}
	and we now conclude by Lemmas \ref{lem:prob.distances} and  \ref{lem:variance.in.balls}.
\end{proof}

\begin{proof}[Proof of Theorem \ref{pro:tilde.u_m}]
We observe that the assertion follows from the following claim:
	There exists a universal constant $C$  such that for all $(X_1,\ldots, X_m) \in \Ocal_m$ and all $m$ sufficiently large
	\begin{align} \label{eq:est.tilde.u_m}
		\| \tilde u_m - u_m\|^2_{\dot H^1(\R^3)} &\leq C \|\nabla (u + G(\rho u - j)) + \nabla \Gamma_m\|^2_{L^2 (\cup_i B_i)} + \nabla \Xi_m\|^2_{L^2(\cup_i B_i)} + Cm^2 \|\Xi_m\|^2_{L^2 (\cup_i B_i)}. \quad ~
	\end{align}

Indeed, accepting the claim for the moment, let $\beta < 1$ and $\eps>0$. Then, using again $u = G (h - \rho u + j)$
	\begin{align*}
		&\P_m \left[m^{\beta} \|\tilde u_m - u_m\|_{\dot H^1(\R^3)} > \eps \right] \\ 
		&\leq \P_m[\Ocal_m^c] + C \eps^{-2} m^{2\beta} \E_m\left[\1_{\Ocal_m} \left(\|\nabla \left(Gh + \Gamma_m  + \Xi_m\right)\|^2_{L^2 (\cup_i B_i)} + m^2 \|\Xi_m\|^2_{L^2 (\cup_i B_i)}\right) \right].
	\end{align*}
Thus, the assertion follows again from Lemmas \ref{lem:prob.distances} and  \ref{lem:variance.in.balls}.
	
	\medskip
	It remains to prove the claim above. It follows from the fact that $u_m - \tilde u_m$ solves the homogeneous Stokes equations outside of the particles.

 	Let $(X_1,\ldots X_m) \in \Ocal_m$. Then, by definition of this set, the balls $B_{2R_m}(X_i)$ are disjoint for $m$ sufficiently large and
 	we may assume in the following that this is satisfied. 
 	
 	By definition of $u_m$ and $\tilde u_m$, we have $-\Delta (\tilde u_m - u_m) + \nabla p = 0$ in $\R^3 \setminus \cup_i \overline{B_i}$. By classical arguments which we include for convenience, this implies
	\begin{align} \label{eq:claim.var.est}
		\|\tilde u_m - u_m\|^2_{\dot H^1(\R^3)} \leq C \left(\|\nabla \tilde u_m\|^2_{L^2 (\cup_i B_i)} + \frac 1 m \sum_i (\tilde u_m -V_i)_i^2\right).
	\end{align}	
	
	Indeed, $\tilde u_m - u_m$ minimizes the $\dot H^1(\R^3)$-norm
 	among all divergence free functions $w$ with $w = \tilde u_m - u_m = \tilde u_m -V_i$ in $\cup_i B_i$.
 	Thus, to show \eqref{eq:claim.var.est}, it suffices to construct a divergence free function $w$ with $w = \tilde u_m -V_i$ in $\cup_i B_i$ such that
 	$\|w\|_{\dot H^1(\R^3)}$ is bounded by the right-hand side of \eqref{eq:claim.var.est}. Since the balls $B_{2R_m}(X_i)$ are disjoint as 
 	$(X_1,\dots,X_m) \in \Ocal_m$, we only need to construct divergence free functions $w_i$ such that $w_i \in H^1_0(B_{2R_m}(X_i))$, $w_i = \tilde u_m - V_i$ in $B_i$ and 
 	\begin{align}
		\|w_i\|^2_{\dot H^1(\R^3)} \leq C \left(\|\nabla \tilde u_m\|^2_{L^2  (B_i)} + \frac 1 m (\tilde u_m - V_i)_i^2\right).
	\end{align}	 
 	It is not difficult to see that such functions $w_i$ exist. For the convenience of the reader, we state this result in Lemma \ref{lem:ext.est} below.
 	Thus, the estimate \eqref{eq:claim.var.est} holds.
 	
 	It remains to prove that the right-hand side of \eqref{eq:claim.var.est} is bounded by the right-hand side of \eqref{eq:est.tilde.u_m}.
 	To this end, let $x \in B_i$ for some $1 \leq i \leq m$. We resort to the definition of $\tilde u_m$ in \eqref{def:tilde.u_m} to deduce, analogously as in \eqref{eq:tilde.u_m(w_i)}, that
 	\begin{align}
 		\tilde u_m(x) &= u(x) - (u)_i + V_i -m^{-\frac{1}{2}} (\xi_m)_i + G(\rho u - j)(x) \\
 		&- G\left[ \sum_{j\neq i} \frac{(u)_j -V_j}{m} \delta^m_j\right](x) - G\left[\sum_{j\neq i} \frac{m^{-\frac{1}{2}} (\xi_m)_j}{m} \delta^m_j \right](x).
 	\end{align}
 	The definitions of $\xi_m$ and $\Theta_m$ from \eqref{def:xi} and \eqref{def:Theta}, the identity $\xi_m = \Theta_m - G \rho \xi_m$ implies that for all $y \in B_i$
 	\begin{align}
 		m^{-\frac{1}{2}} \xi_m(y) = &G(\rho u - j)(y) - G\left[ \sum_{j\neq i} \frac{(u)_j - V_j}{m} \delta^m_j\right](y) - G(\rho m^{-\frac{1}{2}} \xi_m)(y),
 	\end{align}
 	where we used that $(X_1,\dots,X_m) \in \Ocal_m$ to replace $G^m$ by $G$.
 	Thus,
 	\begin{align}
 	 		\tilde u_m(x) - V_i &= u(x) - (u)_i + G(\rho u - j)(x) - (G(\rho u - j))_i + G\left[ \sum_{j\neq i} \frac{(u)_j - V_j}{m} \delta^m_j\right]_i  \\
 	 		 & \quad - G\left[ \sum_{j\neq i} \frac{(u)_j - V_j}{m} \delta^m_j\right](x) +(G(\rho m^{-\frac{1}{2}} \xi_m))_i - G\left[\sum_{j\neq i} \frac{m^{-\frac{1}{2}} (\xi_m)_j}{m} \delta^m_j \right](x)\\
 		&=(u + G(\rho u - j))(x) - (u + G(\rho u - j))_i + \Gamma_m(x) - (\Gamma_m)_i  + \Xi_m(x).
 	\end{align}
 	To conclude the proof, we  again use $(X_1,\dots,X_m) \in \Ocal_m$ to replace $G$ by $G^m$ appropriately.
 	Finally, we combine this identity with
 	\eqref{eq:claim.var.est} and the estimate $(\Xi_m)^2_i \leq C m^3 \|\Xi_m\|^2_{L^2(B_i)}$.
\end{proof}

\begin{lem} \label{lem:ext.est}
	Let $x \in \R^3$, $R > 0$ and $w \in H^1(B_R(x))$ be divergence free. Then, there exists a divergence free function $\varphi \in H^1_0(B_{2 R}(x))$ with $\varphi = w$ in $B_R(x)$ and 
 	\begin{align}
		\|\varphi\|^2_{\dot H^1(\R^3)} \leq C \left(\|\nabla w\|^2_{L^2  (B_R(x))} + R (w)^2_{x,R} \right),
	\end{align}	 
	where $(w)_{x,R} = \fint_{B_R(x)} w \dd x$ and $C$ is a universal constant.
\end{lem}

\begin{proof}
	We write $w = w - (w)_{x,R} + (w)_{x,R}$. By a classical extension result for Sobolev function, there exists $\varphi_1 \in H^1_0(B_{2R}(x))$
	such that $\varphi_1 = w - (w)_{x,R}$ in $B_R(x)$ and
	\begin{align}
		\|\nabla \varphi_1\|_{L^2(\R^3)} \leq C \|\nabla w\|_{L^2  (B_R(x))}.
	\end{align}
	By scaling, the constant $C$ does not depend on $R$.
	
	Furthermore, we take $\varphi_2 = (w)_{x,R} \theta_R$ where $\theta_R \in C_c^\infty(B_{2R}(x))$ is a cut-off function with $\theta_R = 1$ in $B_R(x)$
	and $\|\nabla \theta_R\|_\infty \leq C R^{-1}$.
	Then, 
	\begin{align}
		\|\nabla \varphi_2\|^2_{L^2(\R^3)} \leq C R (w)^2_{x,R}.
	\end{align}	 
	
	Finally, applying a standard Bogovski operator, there exists a function $\varphi_3 \in H^1_0(B_{2r}(x) \setminus B_R(x))$ such that 
	$\dv \varphi_3 = - \dv (\varphi_1 + \varphi_2)$ and
	\begin{align}
		\|\nabla \varphi_3\|_{L^2(\R^3)} \leq C \|\dv(\varphi_1 + \varphi_2)\|_{L^2(\R^3)}.
	\end{align}
	Again, the constant $C$ is independent of $R$ by scaling considerations.
	
	Choosing $\varphi = \varphi_1 + \varphi_2 + \varphi_3$ finishes the proof.
\end{proof}

\section{Proof of probabilistic statements} \label{sec:proof.stochastic}

This section contains the main technical part of the proof of our main result, the probabilistic estimates stated in Proposition \ref{pro:tau_m} and Lemma \ref{lem:variance.in.balls}. The strategy that we will use to estimate  all these terms is to expand the square of sums over the particles and then to use independence of the positions of the particles to calculate the expectations, distinguishing between terms where different particles appear and where one or more particles appear more than once. Then, it will remain to observe that combinatorially relevant terms cancel and that the remaining terms can be bounded sufficiently well, uniformly in $m$. This proof is quite lengthy. Indeed, expanding the square will lead to terms with up to $5$ indices, thus giving rise to a huge number of cases that need to be distinguished.

However, there are only relatively few and basicanalytic tools that we will rely on to obtain these cancellations and estimates. These are collected in the following subsection.
Their proofs are postponed to the appendix.

Some of those estimates concern expressions that will recurrently appear when we take expectations. 
Indeed, since many of the terms in Lemma \ref{lem:variance.in.balls} contain $L^2$-norms in the particles $B_i$, we will often deal with terms of the form
\begin{align}
	\E_m\left [\1_{B^m_{i}}(x) \right] =  \int_{\R^3\times \R^3} \1_{B^m_{y}}(x) f(\de y, \de v) = \int_{\R^3} \1_{B^m_{y}}(x) \rho(y) \dd y  = m^{-3} (\rho)_x.
\end{align}
Another term that recurrently appears due to the definitions of $\tilde u_m$ and $\xi_m$ is 
\begin{align} \label{def:curly.V}
	(\Rcal w)(x) := \E_m\left[ (w)_i \delta^m_i\right](x) &= \int_{\R^3} \rho(y) (w)_y \delta^m_y(x)   \dd y = \fint_{\partial B^m_x} \rho(y) (w)_y \dd y.
\end{align}
To justify this formal computation one tests the expression with a function $\varphi \in C_c^\infty(\R^3)$ and performs some changes of variables.

For the sake of a more compact notation, we introduce
\begin{align} 
	W_i &:= (u)_i - V_i, \label{def:W_i}\\
	F &:= \rho u - j,\label{def:F}  \\
	\Fcal(x) &:= \E_m\left[ W_i \delta^m_i\right](x) = \int_{\R^3\times \R^3}   ((u)_y - v) \delta^m_y(x)   f(\de y,\de v)
	 = \fint_{\partial B^m_x} \rho(y) (u)_y - j(y) \dd \mathcal{H}^2 (y). \label{def:Fcal}
\end{align}

%

\subsection{Some analytic estimates} \label{sec:Auxiliary}
In this subsection, we collect some auxiliary observations and estimates for future reference. All the proofs of the results in this subsection can be found in subsection \ref{subsec:appendix_proofs} of the appendix.

In the following, we denote by $K$ the bounded set defined by
\begin{align} \label{eq:def.K}
	K := \{ x \in \R^3 : \dist(x,\supp \rho) \leq 1\}.
\end{align}
Note  that $B_i \subset K$ almost surely for all $1 \leq i \leq m$ and all $m \geq 1$.

\begin{lem} \label{lem:auxiliary.averages}
	\begin{enumerate}
		\item[(i)] For all $1 \leq p \leq \infty$ and all \(w \in L^p(\R^3)\)
		\begin{align} \label{eq:()}
		\|(w)_\cdot\|_{L^p(\R^3)} \leq \|w\|_{L^p(\R^3)} .
		\end{align}
		\item[(ii)] For all $\alpha >0$, all $1 \leq p \leq \infty$, and all \(w \in L^p(K)\), we have
		\begin{align}
		\|\rho^\alpha(w)_\cdot\|_{L^p(\R^3)}  &\leq C \|w\|_{L^p(K)} \label{eq:V()},
		\end{align}
		where the constant $C$ depends only on $\rho$, $p$ and $\alpha$.
		\item[(iii)] For all \(w\in \dot{H}^1(\R^3)\)
		\begin{align} \label{eq:poincare.average.ball}
		\|w-(w)\|_{L^2(\R^3)}\leq m^{-1} \|w\|_{\dot{H}^1(\R^3)}.
		\end{align}
		\item[(iv)] \label{it:R} The operator $\Rcal$ defined in \eqref{def:curly.V} is a bounded operator from  $L^2(K)$ to $L^2(\R^3) \cap \dot H^{-1}(\R^3)$ and from \(H^1(K)\) to \(H^1(\R^3)\). Moreover, there is a constant \(C\) depending only on \(\rho\) such that
		\begin{align}
		\left\| (\Rcal - \rho) w \right\|_{L^2(\R^3)} & \leq C m^{-1} \|w\|_{H^1(K)} \label{eq:V-V.L^2}, \\
		\norm{(\Rcal - \rho) w}_{\dot H^{-1}(\R^3)} & \leq C m^{-1} \|w\|_{L^2(K)}. \label{eq:V-V.H^-1}
		\end{align}
		\item[(v)] We have
		\begin{align} \label{est_W_i.F}
		\sup_m  \|F\|_{\dot H^{-1}(\R^3)} + \|\Fcal\|_{\dot H^{-1}(\R^3)} + \|F\|_{L^2(\R^3)} + \|\Fcal\|_{L^2(\R^3)} + \E_m[W_1^2] < \infty,
		\end{align}
		and there is a constant \(C\) depending only on \(\rho\) and \(j\) such that
		\begin{align}
		\left\| F-\Fcal \right\|_{L^2(\R^3)} + \norm{F-\Fcal}_{\dot H^{-1}(\R^3)} & \leq C m^{-1}\left(\|u\|_{H^1(K)}+\|j\|_{H^1(\R^3)}\right) \label{eq:F-F.L^2}.
		\end{align}
	\end{enumerate}
	
\end{lem}

\begin{lem} \label{lem:Gdelta}
	There exists a constant $C$ such that for all $x,y \in \R^3$ and all $m \geq 1$, we have
	\begin{align}
		|G \delta^m_y|(x)  &\leq C \frac{1}{|x-y| + m^{-1}}, \label{eq:G.delta.pointwise} \\
		 | A \delta^m_y|(x) &\leq C \left(1 + \frac{1}{|x-y| + m^{-1}} \right), \label{eq:A.delta.pointwise} \\
		|\nabla G \delta^m_y|(x)  &\leq C \frac{1}{|x-y|^2 + m^{-2}}. \label{eq:nablaG.delta.pointwise} 
	\end{align}
	In particular, for any bounded set $K'$ 
	\begin{align}
		\sup_{y \in \R^3}\left( \| G \delta^m_y\|_{L^2(K')} + \|  A \delta^m_y\|_{L^2(K')}  \right)&\leq C(K'). \label{eq:G.delta.L^2_loc}
	\end{align}	
	
	Moreover, for all \(m\geq 1\) and \(y\in \R^3\), it holds 
	\begin{align}
	\|\delta_y^m\|_{\dot H^{-1}(\R^3)} &\leq C m^{1/2}, \label{est:dirac.-1}
	\end{align}
	with a constant independent of $y$ and $m$.
\end{lem}

\begin{lem}\label{lem:delta-delta_m}
	For every \(0\leq s < \frac{1}{2}\) and every bounded set \(K'\)
	\begin{align}\label{eq:GdeltaHs}
		\sup_{y \in \R^3} \|A\delta_y\|_{H^s(K')} + \|G\delta_y\|_{H^s(K')} \leq C_s(K').
	\end{align}
	Furthermore, for every \(0<\eps\leq\frac{1}{2}\) 
	\begin{align}\label{eq:delta-delta_Hs}
		\|\delta_y^m-\delta_y\|_{H^{-3/2-\eps}(K')} \leq C(K')m^{-\eps}.
	\end{align}
\end{lem}

\begin{lem} \label{lem:G-1Gm}
	For any $k \in \N$, $G^m$ is a bounded operator from $\dot H^k(\R^3)$ to $\dot H^{k+2}(\R^3)$. Moreover, there is a constant $C$ that depends only on \(k\) such that
	\begin{align}
		\|G - G^m\|_{\dot H^{k}(\R^3) \to \dot  H^{k}(\R^3)} \leq C m^{-2}, \label{eq:G-G^m.H^k.to.H^k} \\
		\|G - G^m\|_{\dot H^{k}(\R^3) \to \dot  H^{k+1}(\R^3)} \leq C m^{-1}. \label{eq:G-1GmHktoHk+1}
	\end{align}
\end{lem}

\subsection{Proof of Proposition \ref{pro:tau_m}} \label{sec:Proof.tau_m}
For the proof of Proposition \ref{pro:tau_m}, we first introduce another function, $\sigma_m$, intermediate between $\tau_m$ and $\xi_m$.
We first show that $\xi_m$ is close to $\sigma_m$ in the following lemma,
which we will also use in the proof of Lemma \ref{lem:variance.in.balls}.

From now on, we will use the notation \(A\lesssim B\) for scalar quantities \(A\) and \(B\) whenever there is a constant \(C>0\)  such that \(A\leq CB\) and where \(C\) depends neither directly nor indirectly on \(m\).

\begin{lem} \label{lem:xi.rho}
	Using the notation from \eqref{def:W_i} and \eqref{def:F}, let $\sigma_m$ be defined by
	\begin{align} \label{def:sigma_m}
	\sigma_m &:= AG^{-1} \hat \Theta_m, \\
	m^{-1/2} \hat \Theta_m &:= G F - \frac 1 m \sum_{i=1}^m G(W_i \delta^m_i).
\end{align}
Then, for every bounded $K' \subset \R^3$
\begin{align}\label{eq:xi-rho}
	\E_m \left[ \|\xi_m - \sigma_m \|^2_{L^2(K')} \right] \leq C m^{-1}
\end{align}
and
\begin{align}\label{eq:xi-rho.grad}
	\E_m \left[ \|\nabla\xi_m - \nabla\sigma_m \|^2_{L^2(\R^3)} \right] \leq C m.
\end{align}
\end{lem}
\begin{proof}
Let $K$ be the set defined in \eqref{eq:def.K}.
We argue that $A G^{-1}$ satisfies 
\begin{align} \label{eq:AG^-1.L^2}
	\|A G^{-1} w \|_{L^2(K')} \lesssim \|w \|_{L^2(K')}
\end{align}
for any $K' \supset K$ and  any (divergence free) $w \in L^2(K')$.
Indeed, by \eqref{eq:relation.A.G}, we observe that
\begin{align}
	AG^{-1}  = (1 - A \rho) P_\sigma,
\end{align}
and therefore \eqref{eq:AG^-1.L^2} follows from the regularity of $A\rho$ discussed after \eqref{eq:relation.A.G}.

 We recall that both $G$ and $G^m$ (cf. \eqref{eq:g^m.Stokes}) map to divergence free functions.
Thus, by \eqref{eq:AG^-1.L^2}, we have for any bounded set $K' \supset K$
	\begin{align}
		 \E_m \left[\|\xi_m - \sigma_m\|_{L^2(K')}^2\right] 
		 &= \frac 1 m  \E_m \left[ \left\|\sum_i AG^{-1}
		  (G- G^m) (W_i \delta^m_i)\right\|_{L^2(K')}^2\right] \\
		 &\lesssim \frac 1 m  \E_m \left[ \sum_i \sum_{j \neq i} \int_{K'}  \Bigl(AG^{-1}(G - G^m) (W_i \delta^m_i) \Bigr) \Bigl( AG^{-1}(G- G^m) (W_j \delta^m_j) \Bigr)\right] \\
		 &+\frac 1 m  \E_m \left[ \sum_i  \int_{K'}  \left|(G- G^m) (W_i \delta^m_i)\right|^2 \right]\\
		 &=: I_1 + I_2.
	\end{align}
	Recalling the notation \eqref{def:Fcal} and using \eqref{eq:G-G^m.H^k.to.H^k}, we deduce
	\begin{align*}
		I_1 = (m-1)  \|AG^{-1}(G - G^m) \Fcal \|_{L^2(K')}^2 \lesssim (m-1)  \|(G - G^m) \Fcal \|_{L^2(K')}^2  \lesssim  m^{-3}  \| \Fcal \|_{L^2(\R^3)}^2 
		\lesssim m^{-3}
	\end{align*}
	due to \eqref{est_W_i.F}.	
	It remains to bound $I_2$.
	By combining \eqref{eq:G-1GmHktoHk+1} with \eqref{est:dirac.-1}, we obtain
	\begin{align}
		\|(G- G^m)(\delta_{y}^m)\|_{L^2(\R^3)}^2 \lesssim m^{-2} \|\delta_y^m\|_{\dot{H}^{-1}(\R^3)}^2 \lesssim m^{-1}.
	\end{align}
	Thus, by \eqref{est_W_i.F}
	\begin{align}
		I_{2} \lesssim m^{-1} \E_m[W_1]^2 \lesssim m^{-1}.
	\end{align}	
	For the gradient estimate, we can argue similarly: Since  $AG^{-1}$  is bounded from $\dot H^1(\R^3)$ to $\dot H^1(\R^3)$
	\begin{align}
	\E_m \left[\|\nabla(\xi_m - \sigma_m)\|_{L^2(\R^3)}^2\right] 
	&= \frac 1 m  \E_m \left[ \left\|\sum_{i=1}^{m}  \nabla AG^{-1}
	(G- G^m) (W_i \delta^m_i)\right\|_{L^2(\R^3)}^2\right] \\
	&\lesssim \frac 1 m  \E_m \left[ \sum_{i=1}^{m} \sum_{j \neq i} \int_{\R^3}  \Bigl( \nabla AG^{-1}(G - G^m) (W_i \delta^m_i) \Bigr) \Bigl( \nabla AG^{-1}(G- G^m) (W_j \delta^m_j) \Bigr)\right] \\
	&+\frac 1 m  \E_m \left[ \sum_{i=1}^{m}  \int_{\R^3}  \left| \nabla (G- G^m) (W_i \delta^m_i)\right|^2 \right]\\
	&=: I_1 + I_2.
	\end{align}
	Using \eqref{eq:G-1GmHktoHk+1}, we deduce
	\begin{align*}
	I_1 = (m-1)  \|\nabla AG^{-1}(G - G^m) \Fcal \|_{L^2(\R^3)}^2 \lesssim (m-1)  \|\nabla (G - G^m) \Fcal \|_{L^2(\R^3)}^2\lesssim m^{-1}  \| \Fcal \|_{L^2(\R^3)}^2 \lesssim m^{-1}.
	\end{align*}
	It remains to bound $I_2$.
	Using that both $G^m$ and $G$ are bounded operators from $H^{-1}$ to $\dot H^1$, we find with \eqref{est:dirac.-1}
	\begin{align}
	\|\nabla (G- G^m)(\delta_{y}^m)\|_{L^2(\R^3)}^2 \lesssim  \|\delta_y^m\|_{\dot{H}^{-1}(\R^3)}^2 \lesssim m.
	\end{align}
	Thus,
	\begin{align}
	I_{2} \lesssim m \E_m[W_1^2] \lesssim m.
	\end{align}	
	This finishes the proof.
\end{proof}

\begin{cor}\label{cor:xi-sigma_Hs}
	For every \(0\leq s < \frac{1}{2}\) and every \(K'\subset \R^3\) bounded, there is a constant \(C_s(K')>0\) independent of \(m\)	such that
		\[\E_m\left[\|\xi_m-\sigma_m\|_{H^s(K')}^2\right] \leq C_s(K')m^{-1+2s}.\]
\end{cor}

\begin{proof}
	This follows from Lemma \ref{lem:xi.rho} and the interpolation inequality
	\begin{align*}
		\E_m\left[\|\xi_m-\sigma_m\|_{H^s(K')}^2\right] & \lesssim \E_m\left[ \|\xi_m-\sigma_m\|_{L^2(K')}^{2(1-s)}\|\nabla \xi_m- \nabla\sigma_m\|_{L^2(K')}^{2s} \right] \\
		& \leq \E_m\left[ \|\xi_m-\sigma_m\|_{L^2(K')}^{2} \right]^{1-s}\E_m\left[\|\nabla \xi_m- \nabla\sigma_m\|_{L^2(K')}^{2} \right]^s \\
		& \lesssim m^{-1+2s}.
	\end{align*}
	This finishes the proof.
\end{proof}

\begin{proof}[Proof of Proposition \ref{pro:tau_m}]
By Lemma \ref{lem:xi.rho}, it suffices to prove 
\begin{align} \label{eq:sigma-tau.proof}
	\E_m \left[ \|\sigma_m - \tau_m\|^2_{H^s(K')} \right] &\leq C m^{-1+2s} \\
	\E_m[\|\tau_m\|_{\dot{H}^s(K')}^2] &\leq C_s(K')
\end{align}
 for every \(0\leq s < \frac{1}{2}\). 
We introduce $\tilde W_i := u(X_i) - V_i$. It is easily seen that $\E_m[\tilde{W}_1^2] \leq C$ and \(\E_m[|W_1-\tilde{W}_1|] \lesssim \frac{1}{m}\) uniformly in $m$. Since \(\tilde{W}_i\delta_{X_i}\) are independent identically distributed random variables, we obtain

\begin{align*}
	\E_m\left[\|\tau_m\|_{\dot{H}^s(K')}\right] & = \frac{1}{m} \E_m\left[\norm{\sum_{i=1}^{m} AF- A\tilde{W}_i\delta_{X_i}}_{\dot{H}^s(K')}^2\right] \\
	& = \E_m\left[\norm{AF- A\tilde{W}_1\delta_{X_1}}_{\dot{H}^s(K')}^2\right] \\
	& \leq C_s(K')
\end{align*}
by \eqref{eq:GdeltaHs}.

Finally, have to estimate \(\sigma_m-\tau_m\):
\begin{align*}
	 \E_m \left[\|\sigma_m - \tau_m\|_{\dot{H}^s(K')}^2\right] 
	 & = \frac{1}{m} \E_m \left[ \norm{\sum_{i=1}^{m} A\bigl( W_i\delta_i^m - \tilde{W}_i \delta_{X_i}\bigr)}_{\dot{H}^s(K')}^2\right] \\
	 & \leq \frac{1}{m} \sum_{i,j=1}^{m}\E_m \left[ \norm{A\bigl( W_i\delta_i^m - \tilde{W}_i \delta_{X_i}\bigr)}_{\dot{H}^s(K')}\norm{A\bigl( W_j\delta_j^m - \tilde{W}_j \delta_{X_j}\bigr)}_{\dot{H}^s(K')}\right] \\
	 & = \frac{1}{m}\sum_{j \neq i=1}^{m} \E_m \left[ \norm{A\bigl( W_i\delta_i^m - \tilde{W}_i \delta_{X_i}\bigr)}_{\dot{H}^s(K')}\norm{A\bigl( W_j\delta_j^m - \tilde{W}_j \delta_{X_j}\bigr)}_{\dot{H}^s(K')}\right] \\
	 & \quad + \frac{1}{m}\sum_{i=1}^{m} \E_m \left[ \norm{A\bigl( W_i\delta_i^m - \tilde{W}_i \delta_{X_i}\bigr)}_{\dot{H}^s(K')}^2\right] \\
	 & = I_1 + I_2.
\end{align*}
For \(I_1\), notice that by \eqref{eq:F-F.L^2}
\begin{align*}
	I_1 & = (m-1)\|A(F-\tilde{F})\|_{H^s(K')}^2 \\
	& \leq (m-1)\|A(F-\tilde{F})\|_{\dot{H}^1(K')}^{2} \\
	& \leq m^{-1}.
\end{align*}
For \(I_2\), we estimate
\begin{align*}
	\|A(W_i\delta_i^m - \tilde{W}_i\delta_{X_i})\|_{\dot{H}^s(K')} & \leq \norm{A(W_i-\tilde{W}_i)\delta_i^m}_{\dot{H}^s(K')} + \norm{A\tilde{W}_i(\delta_i^m-\delta_{X_i})}_{\dot{H}^s(K')} \\
	& \leq |W_i-\tilde{W}_i| \|A\delta_i^m\|_{\dot{H}^s(K')} + |\tilde{W}_i|\|A(\delta_i^m-\delta_{X_i})\|_{\dot{H}^s(K')} \\
	& \lesssim |W_i-\tilde{W}_i| + m^{s-\frac{1}{2}}|\tilde{W}_i|
\end{align*}
by \eqref{eq:GdeltaHs} and by combining \eqref{eq:delta-delta_Hs} with the fact that \(A\) is a bounded operator from \(\dot{H}^{s-2}(K')\) to \(\dot{H}^s(K')\).
Inserting this above, we find that
\begin{align*}
	I_2 & \lesssim \frac{1}{m} \sum_{i=1}^{m} \E_m\left[\bigl(W_i-\tilde{W}_i| + m^{s-\frac{1}{2}}|\tilde{W}_i|\bigr)^2\bigr)\right] \\
	& \lesssim \E_m\left[|W_i-\tilde{W}_i|^2\right] + m^{-1+2s}\E_m\left[|\tilde{W}_i|^2\bigr)\right] \\
	& \lesssim m^{-1+2s}.
\end{align*}
	Combining the estimates for $I_1$ and $I_2$ yields \eqref{eq:sigma-tau.proof} which finishes the proof.
\end{proof}

\subsection{Proof of Lemma \ref{lem:variance.in.balls}} \label{sec:proof.variance.in.balls}

%
	

We begin the proof of Lemma \ref{lem:variance.in.balls} by observing that we have actually already proved the required estimate for $\Lambda^m$. Indeed, 
$\Lambda^m = m^{-1/2} (\Theta^m - \hat \Theta^m)$ with $\hat \Theta_m$ as in Lemma \ref{lem:xi.rho}. Moreover, in the proof of Lemma \ref{lem:xi.rho}, we showed $\|\Theta^m - \hat \Theta^m\|_{L^2_\loc(\R^3)}^2 \lesssim m^{-1}$.

We divide the rest of proof of Lemma \ref{lem:variance.in.balls} into three steps corresponding to the three terms 
\begin{align}
	I_1 &:= \E_m\left[\1_{\Ocal_m} \|\nabla (u + G (\rho u - j))\|^2_{L^2 (\cup_i B_i)}\right], \\
	I_2 &:= \E_m \left[ \1_{\Ocal_m} \|\nabla \Gamma_m\|^2_{L^2 (\cup_i B_i)} \right], \\
	I_3 &:= m^2 \E_m\left[\1_{\Ocal_m}  \|\Xi_m\|^2_{L^2 (\cup_i B_i)}\right] + \E_m\left[\1_{\Ocal_m}  \|\tilde \Xi_m\|^2_{L^2(K')} \right] + \E_m\left[\1_{\Ocal_m}  \|\nabla \Xi_m\|^2_{L^2(\cup_i B_i)} \right],\qquad \label{eq:I_3}
\end{align}
where $K'$ is a bounded set.
We need to prove $I_k \leq C m^{-2}$ for $k=1,2,3$, uniformly in $m$ with a constant depending only on $h$, $\rho$ and $K'$.
	
\bigskip

\subsubsection*{Step 1: Estimate of $\mathbf{I_1}$.}

\noindent 
Since $\nabla Gh \in L^2(\R^3)$ is deterministic, and the positions of the particles $B_i$ are independent, 
we estimate
\begin{align}
	I_1 = \E_m\left[\1_{\Ocal_m}\|\nabla Gh\|^2_{L^2 (\cup_i B_i)}\right] &\leq \E_m\left[\|\nabla Gh\|^2_{L^2 (\cup_i B_i)}\right] 
	= m^{-2} \int_{\R^3} (\rho)_x |\nabla Gh|^2 \dd x \\
	&\lesssim m^{-2} \|\nabla Gh\|_{L^2(\R^3)}^2 \lesssim m^{-2}.
\end{align}
Here we used \eqref{eq:()} together with  \(\rho\in L^{\infty}(\R^3)\).

\subsubsection*{Step 2: Estimate of $\mathbf{I_2}$.}

\noindent
Since $\Gamma_m$ depends on $m$, the computation is more involved.  
According to the definition of $\Gamma$, we split $I_2$ again. More precisely, it suffices to estimate
\begin{align}
	 I_{2,1} &:=  \E_m\left[\left\| \nabla G\left[\sum_{j \neq i} \frac{(u)_j-V_j}{m} \delta^m_j\right] \right\|^2_{L^2(\cup_i B_i)} \right], \\
	 I_{2,2} &:=  \E_m\left[\| \nabla  G(\rho m^{-\frac{1}{2}} \xi_m) \|^2_{L^2(\cup_i B_i)} \right].
\end{align}
In the first term, we used that for $(Z_1,\dots,Z_m) \in \Ocal_m$ we can replace $G^m$ by $G$ according to \eqref{eq:prop.G^m_lambda}.

We first consider $I_{2,1}$. We expand the square to obtain for any fixed $i$
\begin{align}
	I_{2,1} &= m \E_m\left[\int_{ B_i}  \left( \nabla G\left[ \frac 1 m \sum_{j \neq i} \frac{(u)_j-V_j}{m} \delta^m_j\right]\right)(x)  \left( \nabla G\left[ \frac 1 m \sum_{k \neq i} \frac{(u)_k-V_k}{m} \delta^m_k\right]\right)(x) \right] \\
	&=: \frac 1 m \sum_{j \neq i} \sum_{k \neq i} I_{2,1}^{j,k}.
\end{align}
We distinguish the cases $j\neq k$ and $j=k$.
In the case $j \neq k$, we apply a similar reasoning as for $I_1$: due to the  independence of $Z_i$, $Z_j$, $Z_k$, we have with $\F$ as in \eqref{def:Fcal}
\begin{align} \label{eq:I_21.1}
	I^{jk}_{2,1} & = m^{-4} \int_{\R^3} (\rho)_x \left(\int_{\R^3\times\R^3}\nabla G \left[ \left((u)_y-v\right) \delta^m_y \right](x) f(\de y,\de v)\right)^2 \dd x \\
	& = m^{-4}  \int_{\R^3} (\rho)_x \left(\nabla G[\F](x) \right)^2 \dd x \lesssim m^{-4}\|\nabla G[\F]\|_{L^2(\R^3)}^2, 
\end{align}
where we used again \eqref{eq:()}. Since  by \eqref{est_W_i.F}, $\F$ is bounded in \(\dot{H}^{-1}(\R^3)\), we therefore conclude that
\begin{align}
\sum_{j \neq i} \sum_{k \not \in \{i,k\}} I^{j k}_{2,1} \lesssim  m^{-2}.
\end{align}


%
%


It remains to estimate $I_{2,1}^{jj}$.
We compute
\begin{align}
	I_{2,1}^{jj} &= m^{-4}\int_{\R^3} (\rho)_x \int_{\R^3\times\R^3} \left( \nabla G \left[ \left((u)_y-v\right) \delta^m_y \right](x) \right)^2  f(\de y,\de v) \dd x\\
		&\lesssim m^{-4} \int_{\R^3\times\R^3} \left((u)_y-v\right)^2 
		 \|\nabla G  \delta^m_y \|_{L^2(\R^3)}^2 f(\de y,\de v).
\end{align}
By \eqref{est:dirac.-1}
 \begin{align}
	\|\nabla G  \delta^m_y \|_{L^2(\R^3)}^2 \lesssim  m.
\end{align}
Combining this with \eqref{eq:V()}, we conclude
\begin{align}
	\sum_{j \neq i} I_{2,1}^{jj} & \lesssim m^{-2} \int_{\R^3\times\R^3}  \left((u)_y-v\right)^2 f(\dd y,\dd v) \lesssim m^{-2} \left(\| \rho^{1/2} (u)_\cdot\|_{L^2(\R^3)}^2 + \int_{\R^3\times \R^3} |v|^2 f(\de y, \de v)\right) \\
	& \lesssim m^{-2} \left(\|u\|^2_{L^2(K)} + \int_{\R^3\times \R^3} |v|^2 f(\de y, \de v)\right) \lesssim m^{-2},
\end{align}
by assumption \ref{ass:energy}.

\medskip

We now turn to $I_{2,2}$. We estimate
\begin{align}
	I_{2,2} \leq  \E_m\left[\| \nabla  G(\rho m^{-\frac{1}{2}} \sigma_m) \|^2_{L^2(\cup_i B_i)} \right] +  \E_m\left[\| \nabla  G(\rho m^{-\frac{1}{2}} (\xi_m - \sigma_m) \|^2_{L^2(\cup_i B_i)} \right],
\end{align}
with $\sigma_m$ from Lemma \ref{lem:xi.rho}.
Using this lemma and the fact that $G\rho$ is a bounded operator from $\dot H^1(\R^3)$ to $W^{1,\infty}(\R^3)$, we find
\begin{align}
	 \E_m\left[\| \nabla  G(\rho m^{-\frac{1}{2}} (\xi_m - \sigma_m))\|^2_{L^2(\cup_i B_i)} \right] \lesssim m^{-2}  \| m^{-\frac{1}{2}} (\xi_m - \sigma_m) \|^2_{\dot H^1(\R^3)} \lesssim m^{-2}.
\end{align}

Recalling the definition of $\sigma_m$ from Lemma \ref{lem:xi.rho}, we have
\begin{align}
		 \E_m\left[\| \nabla  G(\rho m^{-\frac{1}{2}} \sigma_m) \|^2_{L^2(\cup_i B_i)} \right] & \leq \sum_{i=1}^{m}\E_m\left[ \left\| \nabla  G\left(\rho A\left[F - \frac{1}{m}\sum_{j=1}^m  W_j \delta^m_j \right]\right) \right\|^2_{L^2(B_i)} \right] \\
		&\lesssim \sum_{i=1}^m\E_m\left[ \left\| \nabla  G\left(\rho A F\right) \right\|^2_{L^2(B_i)} \right] \\
		& + \sum_{i=1}^{m}\E_m\left[ \left\| \nabla  G\left(\rho A\left[\frac{1}{m} \sum_{j=1}^m [W_j \delta^m_j] \right]\right) \right\|^2_{L^2(B_i)} \right]  \\
		&=:  I_{2,2,1} +  I_{2,2,2}.
\end{align}
This is a very rough estimate, since we actually expect cancellations from the difference. However, these cancellations are not needed here for the desired bound. Indeed, since $G \rho A$ is a bounded operators from \(\dot{H}^{-1}(\R^3)\) to $\dot H^1 (\R^3)$, $I_{2,2,1}$ is controlled analogously as $I_1$.

It remains to estimate $I_{2,2,2}$. We expand the square again and write
\begin{align}
	I_{2,2,2} &= \sum_{i=1}^{m}\E_m \left[ 
	\int_{B_i} \left( \nabla  G	\left(\rho A\left[\frac{1}{m} \sum_{j=1}^{m} W_j \delta^m_j \right]\right) 	\right) \cdot \left( \nabla  G\left(\rho A\left[\frac{1}{m} \sum_{k=1}^{m}  W_k \delta^m_k\right]\right) \right) \dd x 
	\right] \\
	&=:  \sum_{i=1}^{m} \sum_{j=1}^{m}\sum_{k=1}^{m} I^{i,j,k}_{2,2,2}.
\end{align}

We have to distinguish the cases where all \(i,j,k\) are distinct, the case where \(j=k\) but \(j\neq i\), the case where \(i=j\) or \(i=k\) but \(j\neq k\), and, finally, the case where \(i=j=k\).

In the first case, we can proceed analogously as for $I^{j,k}_{2,1}$. In particular, we use the definition of \(\F\) to deduce
\begin{align}
	\sum_{i=1}^{m} \sum_{j\neq i}\sum_{k \not \in \{i,j\}}  I^{i,j,k}_{2,2,2} 
	&= m^{-3} \frac{m(m-1)(m-2)}{m^2} \int_{\R^3} (\rho)_x (\nabla  G \rho A  \F)^2 \dd x \\
	& \lesssim m^{-2} 
	\|\nabla  G \rho A \F \|^2_{L^2(\R^3)} 
	 \lesssim  m^{-2}  \| \F\|^2_{\dot H^{-1}(\R^3)} 
	\lesssim  m^{-2},
\end{align}
since  $ G \rho A$ is also bounded from $\dot H^{-1}(\R^3)$ to $\dot H^1(\R^3)$. 

Next, we estimate \(I_{2,2,2}^{i,j,j}\). Analogously as for $I_{2,1}^{j,j}$, we obtain
\begin{align}
	\sum_{i=1}^{m} \sum_{j\neq i} I^{i,j,j}_{2,2,2} &= m^{-3}\frac{m(m-1)}{m^2} \int_{\R^3} (\rho)_x \int_{\R^3} \left( \nabla  G\rho A ((u)_y-v) \delta^m_y(x)  \right)^2 f(\de y,\de v) \dd x\\
	&\lesssim  m^{-3} \int_{\R^3} ((u)_y - v)^2 \left\|\nabla  G \rho A\delta^m_y (x)\right\|_{L^2_x(\R^3)}^2 f(\de y,\de v).
\end{align}
Since \(\nabla G V\) is a bounded operator from \(\dot{H}^1(\R^3)\) to \(L^2(\R^3)\), we obtain by \eqref{est:dirac.-1} combined with \eqref{eq:V()} and using \ref{ass:energy}
\begin{align}
	\sum_{i=1}^{m} \sum_{j\neq i} I^{i,j,j}_{2,2,2} & \lesssim m^{-2}\left(\|\rho^{1/2}(u)\|_{L^2(\R^3)}^2 + \int_{\R^3\times\R^3} |v|^2 f(\de y, \de v) \right) \lesssim m^{-2}.
\end{align}

The third estimate concerns \(I^{i,i,k}_{2,2,2}\). By  symmetry, \(I^{i,j,i}_{2,2,2}\) is dealt with analogously. We have, using \eqref{est:dirac.-1}, \eqref{est_W_i.F}, and \eqref{eq:V()} together with \eqref{eq:()},
\begin{align}
	\sum_{i=1}^{m} \sum_{k\neq i}  I^{i,i,k}_{2,2,2} &=  \frac{m(m-1)}{m^2} \int_{\R^3} \E_m\left[\1_{B_i} \nabla G \left(\rho A \left[W_i\delta_i^m\right]\right)\right]\nabla  G \left(\rho A \left[\Fcal\right]\right) \dd x \\
	& \lesssim \|\nabla G \rho A  \mathcal{F}\|_{L^2(\R^3)} \left\|\int_{\R^3\times\R^3} \1_{B^m(y)} \nabla G \left(\rho A \left[ ((u)_y-v)\delta_y^m\right]\right)f(\de y, \de v) \right\|_{L^2(\R^3)} \\
	& \lesssim  \sup_{y\in\R^3}\|\nabla G\rho A \delta_y^m\|_{L^{\infty}} \left\| \int_{\R^3\times\R^3} ((u)_y-v)\1_{B^m(y)}f(\de y,\de v)\right\|_{L^2(\R^3)} \\
	& \lesssim m^{1/2} m^{-3}\|(\rho(u)_\cdot-j)_\cdot\|_{L^2(\R^3)} \lesssim m^{-5/2}.
\end{align}
We also used that the operator \(\nabla G\rho A\) maps \(\dot{H}^{-1}(\R^3)\) into \(L^{\infty}(\R^3)\), as well as \(j\in L^2(\R^3)\) by assumption \ref{ass:j}.

Finally, we estimate \(I^{i,i,i}_{2,2,2}\). Using \eqref{est:dirac.-1} and \eqref{eq:V()}, we obtain

\begin{align}
	\sum_{i=1}^m I^{i,i,i}_{2,2,2} & = \frac{m}{m^2} \int_{\R^3} \E_m\left[\1_{B_i} \left|\nabla G \left(\rho A \left[W_i\delta_i^m\right]\right)\right|^2)\right] \dd x \\
	& = \frac 1 m \int_{\R^3} \int_{\R^3\times \R^3} \1_{B^m(y)}\left|\nabla G \left(\rho A \left[((u)_y-v)\delta_y^m\right]\right)\right|^2 f(\de y, \de v) \dd x \\
	& \lesssim \frac{1}{m}\sup_{y\in\R^3} \|\nabla G \rho A \delta_y^m\|_{L^{\infty}(\R^3)}^2 \int_{\R^3} \int_{\R^3\times \R^3} \1_{B^m(y)} \left((u)_y-v\right)^2 f(\de y, \de v) \dd x \\
	& \lesssim \int_{\R^3} \int_{\R^3\times\R^3} \1_{B^m(y)} \left(|(u)_y|^2 + |v|^2\right) f(\de y, \de v) \dd x \\
	& \lesssim m^{-3} \left(\int_{\R^3} \rho(y)|(u)_y|^2 \dd y \dd x + \int_{\R^3\times\R^3}|v|^2 f(\de y, \de v)\right) \\
	& \lesssim m^{-3}.
\end{align}
	
This finishes the estimate of $I_{2,2,2}$. Therefore, the estimate of $I_{2,2}$ is complete, which also finishes the estimate of $I_2$.
	
\bigskip	
\subsubsection*{Step 3: Estimate of $\mathbf{I_3}$.}

We recall from \eqref{eq:I_3} that $I_3$ consists of three terms, which we denote by $J_1, J_2$ and $J_3$.
We will focus on the proof on $J_1$ as this is the most difficult term. We will comment on the adjustments
needed to treat $J_2$ and $J_3$ along the estimates for $J_1$. Roughly speaking, the main difference between $J_1$ and $J_2$ is that one considers $L^2(\cup_i B_i)$ for $J_1$ and $L^2_\loc(\R^3)$ for $J_2$. Naively, $J_1$ should therefore be better by a factor $|\cup_i B_i| \sim m^{-2}$,
which is exactly the estimate we obtain. Moreover, $J_3$ concerns the gradient of the terms in $J_1$. Since we may loose a factor $m^{-2}$ going
from $J_1$ to $J_3$, it will not be difficult to adapt the estimates for $J_1$ to $J_3$ using the  gradient estimates in Section \ref{sec:Auxiliary}. For the sake of completeness we detail the estimates for $J_3$ in the appendix.

\subsubsection*{Step 3.1: Expansion of the terms}

As in the previous step, we first want to replace all occurrences of $G^m$  by $G$. Note that $G^m$ is present both explicitly in the definition of $\Xi^m$ and also implicitly through $\xi_m$.
By \eqref{eq:prop.G^m_lambda} and independence of the position of the particles, it holds
\begin{align} \label{eq:J_1.0}
\begin{aligned}
	&  m^2 \E_m\left[ \1_{\Ocal_m}  \|\Xi_m\|^2_{L^2 (\cup_i B_i)}\right] \\
	& \leq m^2 \E_m\left[\1_{\Ocal_m}\sum_{i=1}^{m}\int_{B_i} \left| G (\rho m^{-\frac{1}{2}} \xi_m) - G\left[\sum_{j \neq i} \frac{m^{-\frac{1}{2}} (\xi_m)_j}{m} \delta^m_j \right]\right|^2 \dd x \right] \\
	& = m^3\E_m\left[\1_{\Ocal_m}\int_{B_i} \left| G (\rho m^{-\frac{1}{2}} \xi_m) - G\left[\sum_{j \neq i} \frac{m^{-\frac{1}{2}} (\xi_m)_j}{m} \delta^m_j \right]\right|^2 \dd x \right] \\
	& \lesssim m^3 \E_m\left[ \int_{B_i}\left|G(\rho m^{-1/2}(\xi_m-\sigma_m))\right|^2\de x \right]\\
	& + m^3 \E_m\left[\1_{\Ocal_m}\int_{B_i} \left| G (V m^{-\frac{1}{2}}\sigma_m) - G\left[\sum_{j \neq i} \frac{m^{-\frac{1}{2}} (\xi_m)_j}{m} \delta^m_j \right]\right|^2 \dd x \right],
\end{aligned}
\end{align}
where on the right-hand side, $i$ is any of the $m$ identically distributed particles.
We use that \(G\rho\) is a bounded operator from \(L^2(K)\) to \(L^{\infty}(B_i)\) and Lemma \ref{lem:xi.rho} to deduce
\begin{align}
	m^3 \E_m\left[ \int_{B_i}\left|G(\rho m^{-1/2}(\xi_m-\sigma_m))\right|^2 \right] & \lesssim  \E_m\left[ \norm{G(\rho m^{-1/2}(\xi_m-\sigma_m))}_{L^{\infty}(B_i)}^2 \right] \\
	& \lesssim   m \E_m\left[ \norm{m^{-1/2}(\xi_m-\sigma_m)}_{L^{2}(K)}^2 \right]  \\
	& \lesssim m^{-2}.
\end{align}
This implies, that for the estimate of $J_1$, it suffices to show that
\begin{align}
	\mathfrak{J}_1 := \E_m\left[\1_{\Ocal_m}\int_{B_i} \left| G (\rho m^{-\frac{1}{2}}\sigma_m) - G\left[\sum_{j \neq i} \frac{m^{-\frac{1}{2}} (\xi_m)_j}{m} \delta^m_j \right]\right|^2 \dd x \right] \lesssim m^{-5}.
\end{align}

By the definitions of \(m^{-\frac{1}{2}}\xi_m\) and \(m^{-\frac{1}{2}}\rho_m\) (cf. \eqref{def:xi} and \eqref{def:sigma_m}) together with \eqref{eq:prop.G^m_lambda}, we have in $\Ocal_m$
	\begin{align}
	G (\rho m^{-\frac{1}{2}}\sigma_m) - G\left[\sum_{j \neq i} \frac{m^{-\frac{1}{2}} (\xi_m)_j}{m} \delta^m_j \right] = \frac 1  m \sum_{k=1}^m \sum_{j=1}^m \Psi_{jk}, \\ 
	\label{eq:def.Psi}
	\Psi_{j,k}(x)  := G\left[\rho A \left(F - W_k\delta^m_k\right) \right]  - (1-\delta_{ij})G\left[ \left(A\left(F - (1-\delta_{jk})W_k\delta^m_k\right) \right)_j \delta^m_j  \right].
	\end{align}
	(Strictly speaking  $\Psi_{j,k}$ depends on $i$, but we omit this dependence for the ease of notation.)

Thus,	
\begin{align}
	\mathfrak{J}_1 
	& \leq m^{-1} \sum_{j=1}^{m} \sum_{k=1}^{m} \sum_{n=1}^{m} \sum_{\ell=1}^{m} I_3^{i,j,k,n,\ell}, \\
	I_3^{i,j,k,n,\ell} &:= \E_m\left[\int_{B_i} \Psi_{j,k}(x) \Psi_{n,\ell}(x) \dd x\right].
\end{align}
Similarly, we have the estimate
\begin{align}
	J_3 &\lesssim \E_m\left[\int_{\cup_iB_i}\left|\nabla G(\rho m^{-\frac{1}{2}}(\xi_m-\sigma_m))\right|^2\right] + \mathfrak{J}_3 \lesssim m^{-2} + \mathfrak{J}_3, \\
	\mathfrak{J}_3 &:= m^{-3} \sum_{j=1}^{m} \sum_{k=1}^{m} \sum_{n=1}^{m} \sum_{\ell=1}^{m} \E_m\left[\int_{B_i} \nabla \Psi_{j,k}(x) \nabla \Psi_{n,\ell}(x) \dd x\right], \label{def:frak.J_2}
\end{align}
with the same proof as before using that \(\nabla G\rho\) is a bounded operator from \(\dot{H}^1(\R^3)\) to \(W^{1,\infty}(\R^3)\) and the second part of Lemma \ref{lem:xi.rho}. 

Furthermore,
\begin{align}
J_2 &\lesssim \E_m\left[\left\| G(\rho m^{-\frac{1}{2}}(\xi_m-\sigma_m)) \right\|_{L^2(K')}^2\right] + \mathfrak{J}_2 \lesssim m^{-2} + \mathfrak{J}_2, \\
	\mathfrak{J}_2 &:= m^{-4} \sum_{j=1}^{m} \sum_{k=1}^{m} \sum_{n=1}^{m} \sum_{\ell=1}^{m} \int_{K'}  \E_m\left[\tilde \Psi_{j,k}(x) \tilde \Psi_{n,\ell}(x)\right] \dd x,
\end{align}
where $\tilde \Psi_{j,k}$ denotes the function that is obtained by omitting the factor $(1-\delta_{ij})$ in \eqref{eq:def.Psi}.

\medskip
Relying on this structure enables us to make more precise the argument why the estimate for $\J_1$ is most difficult compared to $\J_2$ and $\J_3$.
Indeed, for the estimate for $\J_3$, one just follows the same argument as for $\J_1$. The relevant estimates in Section \ref{sec:Auxiliary}
show that whenever $\nabla G$ instead of $G$ appears, we loose (at most) a factor $m^{-1}$. For completeness, we provide the proof of the estimates regarding $\J_3$ in the appendix.

On the other hand, for $\J_2$, we can use the estimates that we will prove for the terms of $\J_1$ in the case when the index $i$ is different from all the other indices.
Indeed, in those cases, $\Psi_{j,k} = \tilde \Psi_{j,k}$, and we will always estimate
\begin{align}
	|I_3^{i,j,k,n,\ell}| &=  \left| m^{-3}\int_{\R^3} (\rho)_x \E_m\left[ \Psi_{j,k}  \Psi_{n,\ell}\right]  \dd x \right| \lesssim m^{-3} \left\|\E_m\left[  \Psi_{j,k}(x) \Psi_{n,\ell}(x) \right]  \right\|_{L^1_\loc(\R^3)}.
\end{align}
Thus, the bound for $\J_2$ is a direct consequence of the estimates we will derive to bound $\J_1$.

\medskip

Recall that we need to prove \(|\J_1|\lesssim m^{-2}\).
We will split the sum into the cases \(\#\{i,j,k,n,\ell\}=\alpha\), \(\alpha=1,\ldots,5\). Then, since \(i\) is fixed, there will be \(m^{\alpha-1}\) summands for the case \(\#\{i,j,k,n,\ell\}=\alpha\). Thus, it is enough to show that in each of these cases
	\[|I_3^{i,j,k,n,\ell}| \lesssim m^{-\alpha},\quad\alpha = \#\{i,j,k,n,\ell\}.\]
	
To prove this estimate, we have to rely on cancellations between the terms that $\Psi_{j,k}$ is composed of. To this end, we denote the first part of \(\Psi_{j,k}\) by
\begin{align}
	\Psi_{k}^{(1)} := \Psi^{(1,1)} + \Psi_{k}^{(1,2)} :=G\left[\rho A F - \rho A[W_k\delta_k^m] \right],
\end{align}
and the second part by
\begin{align}
	\Psi_{j,k}^{(2)} := \Psi_{j}^{(2,1)} + \Psi_{j,k}^{(2,2)}:= (1-\delta_{ij})G\left[ \left(A\left(F - (1-\delta_{jk})W_k\delta^m_k\right) \right)_j \delta^m_j  \right].
\end{align}

We observe that
\begin{equation}
\label{eq:exp.VAVu}
\begin{aligned}
\E_m[\Psi^{(1,1)}] &= G  \rho  A F, \\
\E_m[\Psi^{(1,2)}_k] &= G  \rho  A \Fcal, \\
\E_m[\Psi^{(2,1)}_j] &= (1- \delta_{ij}) G  \Rcal  A F, \\
\E_m[\Psi^{(2,2)}_{j,k}] &= (1- \delta_{ij})(1-\delta_{jk}) G  \Rcal  A  \Fcal.
\end{aligned}
\end{equation}

\bigskip

\subsubsection*{Step 3.2: The cases in which at most $\mathbf{2}$ indices are equal}

In many cases, we can rely on cancellations within $\Psi_{k}^{(1)}$ and $\Psi_{j,k}^{(2)}$.
Indeed, we will prove the following lemma:
\begin{lem} \label{lem:cancellations.not.cross}
	Let $K' \subset \R^3$ be bounded. Then, 
	\begin{align}
		\left\|\E_m[\Psi_{k}^{(1)}] \right\|_{L^2(K')} &\lesssim m^{-1}, \label{eq:Psi^1} \\
		\left\|\E_m[\Psi_{j,k}^{(2)}] \right\|_{L^2(K')} &\lesssim m^{-1} \quad \text{if } j \neq k \label{eq:Psi^2}.
	\end{align}
\end{lem}

There are only three cases (up to symmetry), where we have to rely on cancellations between $\Psi_{k}^{(1)}$ and $\Psi_{j,k}^{(2)}$ to estimate $I_3^{i,j,k,n,\ell}$.
These are the cross terms, when either $j=n$, or $k= \ell$, or $j=\ell$, and all the other indices are different.
In these cases, we will rely on the following lemma:
\begin{lem} \label{lem:cancellations.cross}
	Let $K' \subset \R^3$ be bounded. Then,
\begin{align}
		\label{eq:jn} 
			\|\E_m\left[\Psi_{j,k} \Psi_{j,\ell}\right] \|_{L^1(K')}  \lesssim m^{-2} \quad \text{if } \#\{i,j,k,\ell\} = 4, \\
		\label{eq:kl} 
			\|\E_m\left[\Psi_{j,k} \Psi_{n,k}\right]\|_{L^1(K')}  \lesssim m^{-2} \quad \text{if } \#\{i,j,k,n\} = 4, \\
		\label{eq:jl} 
			\|\E_m\left[\Psi_{j,k} \Psi_{n,j}\right]\|_{L^1(K')} \lesssim m^{-2} \quad \text{if } \#\{i,j,k,n\} = 4.
	\end{align}
\end{lem}

Finally, we obtain the following estimates, useful in particular for the cases in which $i=k$.

\begin{lem} \label{lem:bounds.not.cross}
Let $K' \subset \R^3$ be bounded. Then, for any $i,j,k$,
	\begin{align}
		&\left\|\E_m[\Psi^{(1,1)}] \right\|_{L^2(K')} 
		+ \left\|\E_m[\Psi_{k}^{(1,2)}] \right\|_{L^2(K')}+ \left\|\E_m[\Psi_{j}^{(2,1)}] \right\|_{L^2(K')} 
		+ \left\|\E_m[\Psi_{j,k}^{(2,2)}] \right\|_{L^2(K')} \lesssim 1. \label{eq:bounds.separate.no.i}\\	
		&\left\|\E_m[\1_{B^m_i}\Psi^{(1,1)}] \right\|_{L^2(\R^3)} 
		+ \left\|\E_m[\1_{B^m_i}\Psi_{k}^{(1,2)}] \right\|_{L^2(\R^3)}+ \left\|\E_m[\1_{B^m_i}\Psi_{j}^{(2,1)}] \right\|_{L^2(\R^3)} 
		+ \left\|\E_m[\1_{B^m_i}\Psi_{k,j}^{(2,2)}] \right\|_{L^2(\R^3)} \\
		&\lesssim m^{-3}. \label{eq:bounds.separate.with.i}\\
	\end{align}
\end{lem}

Combining these lemmas allows us to estimate $I_3^{i,j,k,n,\ell}$ in all the cases when \(\alpha = \#\{i,j,k,n,\ell\} \geq 4\).

\begin{cor}
The following estimates hold true where the implicit constants are independent of \(m\):
\begin{enumerate}
	\item If \(\#\{i,j,k,n,\ell\} = 5\), then
		\[|I_3^{i,j,k,n,\ell}| \lesssim m^{-5}.\]
	\item If \(\#\{i,j,k,n,\ell\} = 4\), then
			\begin{align}
				|I_3^{i,j,k,n,\ell}| \lesssim m^{-4}. \label{eq:alpha=4}
			\end{align}
\end{enumerate}
\end{cor}
\begin{proof}
	If \(\#\{i,j,k,n,\ell\} = 5\), then by independence, the Hölder inequality
	and Lemma \ref{lem:cancellations.not.cross}
	\begin{align*}
		\abs{I_3^{i,j,k,n,\ell}} 
		& \leq \norm{\E_m\left[\textbf{1}_{B_{\frac{1}{m}}(w_i)}\right]}_{L^{\infty}(\R^3)} \norm{\E_m\left[\Psi_{j,k}\right]}_{L^2(K)} \norm{\E_m\left[\Psi_{n,\ell} \right]}_{L^2(K)} \\
		& \lesssim m^{-3}m^{-1}m^{-1} =m^{-5}.
	\end{align*}

	If \(\#\{i,j,k,n,\ell\} = 4\), we need to distinguish all the possible combinations of two indices being equal.
	Depending on which indices coincide, we split the product by independence of the other indices. If $j=n$, $k=\ell$ or $j= \ell$ (or $k=n$ which is the same), we rely on Lemma \ref{lem:cancellations.cross} and gain an additional factor $m^{-3}$ from the expectation of $\1_{B^m_i}$.
	
	If $j=k$ (or analogously $n= \ell$), the expectation completely factorizes into 
	$\E_m[\1_{B^m_i}] \E_m[ \Psi_{jj}] \E_m[\Psi_{n\ell}]$ and we can apply
	\eqref{eq:bounds.separate.no.i} for the second factor and Lemma \ref{lem:cancellations.not.cross} for the third factor.
	
	Finally, in all the other cases we can, without loss of generality, split the expectation into $\E_m[\1_{B^m_i} \Psi_{jk}] \E_m[\Psi_{n\ell}]$
	and apply \eqref{eq:bounds.separate.with.i} for the first factor and Lemma \ref{lem:cancellations.not.cross} for the second factor.
\end{proof}

We finish this step by giving the proofs of Lemmas \ref{lem:cancellations.not.cross}, \ref{lem:cancellations.cross} and \ref{lem:bounds.not.cross}.

\begin{proof}[Proof of Lemma \ref{lem:cancellations.not.cross}]
	By \eqref{eq:exp.VAVu}, we have
	\begin{align*}
		\E_m[\Psi^{(1)}_{k}] = G \rho A(F - \Fcal),
	\end{align*}
	and using \eqref{eq:F-F.L^2} yields \eqref{eq:Psi^1}.
	Similarly, for $j \neq k$, $i\neq j$,
	\begin{align*}
		\E_m[\Psi^{(2)}_{j,k}]  = G \Rcal A  (F - \Fcal).
	\end{align*}
	Using again \eqref{eq:F-F.L^2} and recalling from Lemma \ref{lem:auxiliary.averages}  that $\mathcal R$ is a bounded operator from $L^2(K)$ to $\dot H^{-1}(\R^3)$  yields \eqref{eq:Psi^2}. For $i=j$, $\Psi^{(2)}_{j,k} = 0$ and there is nothing to prove. 
\end{proof}

\begin{proof}[Proof of Lemma \ref{lem:cancellations.cross}]
	Regarding \eqref{eq:jn}, we have
	\begin{align*} 
		\E_m\left[\Psi_{j,k} \Psi_{j,\ell}\right] 
		& = \iiint \left(\vphantom{\sum_{i=1}^{m}}G\left[\rho A\left(F - \bigl((u)_{y_2}-v_2\bigr)\delta^m_{y_2} \right) \right]- G\left[ \bigg(A\left(F - \bigl((u)_{y_2} -v_2\bigr)\delta^m_{y_2}\right) \bigg)_{y_1} \delta^m_{y_1} \right]\right) \\
		 & \qquad \qquad \left(\vphantom{\sum_{i=1}^{m}}G\left[\rho A \left(F -\bigl((u)_{y_3} -v_3\bigr)\delta^m_{y_3}\right) \right]- G\left[ \bigg(A\left(F - \bigl((u)_{y_3} -v_3\bigr)\delta^m_{y_3}\right) \bigg)_{y_1} \delta^m_{y_1} \right]\right) \\
		& \qquad \qquad \qquad \qquad f(\de y_1,\de v_1) f(\de y_2,\de v_2)f(\de y_3,\de v_3)\\
		& = \int \rho(y_1) \left( G\rho A(F - \Fcal) -  (A(F - \Fcal))_{y_1} G \delta^m_{y_1}  \right)^2 \dd y_1.
	\end{align*}
	We obtain
	\begin{align*}
		 \|\E_m\left[\Psi_{j,k} \Psi_{j,\ell}\right] \|_{L^1(K')} 
		& \lesssim \|G\rho A(F - \Fcal)\|_{L^2(K')}^2 + \int \rho(y) (A(F-\Fcal))^2_y \|G \delta^m_y\|_{L^2(K)}^2 \dd y\\
		& \lesssim  m^{-2} + \|A(F -  \Fcal)\|^2_{L^2(K')}  \lesssim m^{-2},
	\end{align*}
		where we used \eqref{eq:F-F.L^2} for both terms and \eqref{eq:G.delta.L^2_loc} and \eqref{eq:V()} for the second term.
	
	\medskip

	Regarding \eqref{eq:kl}, we compute
	\begin{align*} 
	 \E_m\left[\Psi_{j,k} \Psi_{n,k}\right]  &= \iiint \left(\vphantom{\sum_{i=1}^{m}}G\left[\rho A\left(F -\bigl((u)_{y_2} - v_2\bigr)\delta^m_{y_2} \right) \right]- G\left[ \bigg(A\left(F -\bigl((u)_{y_2} - v_2\bigr)\delta^m_{y_2}\right) \bigg)_{y_1} \delta^m_{y_1} \right]\right) \\
		& \qquad \qquad \left(\vphantom{\sum_{i=1}^{m}}G\left[\rho A\left(F -\bigl((u)_{y_2} - v_2\bigr)\delta^m_{y_2} \right) \right]- G\left[ \bigg(A\left(F -\bigl((u)_{y_2} - v_2\bigr)\delta^m_{y_2}\right) \bigg)_{y_3} \delta^m_{y_3} \right]\right) \\
	&	\qquad \qquad \qquad \qquad f(\de y_1,\de v_1) f(\de y_2,\de v_2)f(\de y_3,\de v_3) \\
		& = \int \rho(y_2) \left(G(\rho - \Rcal)A F - \bigl((u)_{y_2}-v_2\bigr) G(\rho - \Rcal)A  \delta^m_{y_2} \right)^2 f(\de y_2,\de v_2).
	\end{align*}	
	Thus, we obtain 
	\begin{align*}
		\|\E_m\left[\Psi_{j,k} \Psi_{n,k}\right] \|_{L^1(K')} 
		&\lesssim   \|G(\rho - \Rcal)AF\|_{L^2(K')}^2 + \sup_z \|G(\rho - \Rcal)A \delta^m_z\|^2_{L^2(K')}\int ((u)_z-v)^2  f(\de z,\de v) \\
		& \lesssim m^{-2},
	\end{align*}
	where we used \eqref{eq:G.delta.L^2_loc} for both terms and \eqref{eq:V()} and \ref{ass:energy} for the second term.
	
	\medskip
	
	Finally, to prove \eqref{eq:jl}, we just apply Young's inequality to reduce to the previous two estimates. Indeed,
	\begin{align*} 
		\E_m\left[\Psi_{j,k} \Psi_{n,j}\right] 
		&= \int \left( G \rho A (F - \Fcal) -  \left(A (F - \Fcal)u\right)_{y} G \delta^m_{y}\right) \\
		& \qquad   \left(G(\rho - \Rcal)A F - \bigl( (u)_{y}-v\bigr) G(\rho-\Rcal)A \delta^m_y \right) f(\de y,\de v)\\
		& \leq  \int \rho(y)\left( G \rho A (F - \Fcal) -  \left(A (F - \Fcal)u\right)_{y} G \delta^m_{y}\right)^2 \dd y \\
		& +  \int \left(G(\rho - \Rcal)A F - \bigl( (u)_{y}-v\bigr) G(\rho-\Rcal)A \delta^m \right)^2 f(\de y,\de v).
	\end{align*}
	These two terms are exactly the ones we have estimated in the previous two steps.
\end{proof}

\begin{proof}[Proof of Lemma \ref{lem:bounds.not.cross}]
	The first estimate, \eqref{eq:bounds.separate.no.i}, follows directly from \eqref{eq:exp.VAVu} and \eqref{est_W_i.F} together with the fact that the operators
	$G \rho A $, $G \rho A$,  $G \Rcal A$ and $G \Rcal A$ are all bounded from $\dot H^1(\R^3)$ to $L^2_\loc(\R^3)$.
	
	\medskip
	
	Regarding \eqref{eq:bounds.separate.with.i}, we first observe that these estimates follow directly from \eqref{eq:bounds.separate.no.i}
	in the cases, when $i \neq k$. Indeed,  if $i$ is different from both $j$ and $k$, the expectation factorizes. Moreover, the case $i=j$ is trivial, since the terms with index $j$ vanish for $i=j$.
	
	If $i=k$, we only need to consider those terms, where $k$ appears, i.e. $\Psi^{(1,2)}_k$ and $\Psi^{(2,2)}_{j,k}$.
	Again, we only need to consider the case $j \neq k = i$.
	
	 We have for $\Psi^{(1,2)}_k$
	\begin{align}
		\|\E_m [\1_{B^m_i} \Psi^{(1,2)}_{i}]\|_{L^2(\R^3)} & = \left \| \int \1_{B^m(y)} G \rho A \left[ \bigl((u)_y - v\bigr) \delta^m_y\right] f(\de y,\de v) \right\|_{L^2(\R^3)} \\
		&\leq \sup_{y \in \R^3}\|G\rho A\delta^m_y \|_{L^\infty(\R^3)}  \left\|\int  \bigl((u)_y-v\bigr) \1_{B^m(y)} f(\de y, \de v) \right\|_{L^2(\R^3)} \\
		& \lesssim m^{-3} \| (\rho (u)_\cdot -j)_\cdot \|_{L^2(\R^3)} \lesssim m^{-3},
		\end{align}
	where we used \eqref{eq:G.delta.L^2_loc}, \eqref{eq:()} and \eqref{eq:V()}.
	Since for $j \neq i$, 
	\begin{align}
	\E_m [\1_{B^m_i} \Psi^{(2,2)}_{j,i}] & = \int \1_{B^m(y)} G \Rcal A\left[ \bigl((u)_y-v\bigr) \delta^m_y\right] f(\de y, \de v),
	\end{align}
	the estimate of this term is analogous.
\end{proof}

\subsubsection*{Step 3.3: The cases in which the number of different indices is  $\mathbf{3}$ or less.}

It remains to estimate $|I_3^{i,j,k,n,\ell}|$, when $\# \{i,j,k,n,\ell\} \leq 3$. We will show that
$|I_3^{i,j,k,n,\ell}| \lesssim m^{-3}$ for $\# \{i,j,k,n,\ell\} = 3$, and $|I_3^{i,j,k,n,l}| \lesssim m^{-2}$ for $\# \{i,j,k,n,\ell\} \leq 2$.
Formally, a factor $m^{-3}$ can be expected to come from the term $\1_{B^m_i}$, so that cancellations are not needed for the estimates of those term.
We will see that this strategy works for all the terms except for $I_3^{i,j,i,j,\ell}$ with $i,j,\ell$ mutually distinct.

\medskip

Thus, in all cases except $I_3^{i,j,i,j,\ell}$ with $i,j,\ell$ mutually distinct, we just brutally estimate the product $ \Psi_{j,k} \Psi_{n,\ell}$ via the triangle inequality
\begin{align}
	\left|I_3^{i,j,k,n,\ell}\right| \leq \sum_{\alpha,\beta,\gamma,\delta =1}^2 \int \left|\E_m\left[ \1_{B^m_i} \Psi_{j,k}^{(\alpha,\beta)} \Psi_{n,\ell}^{(\gamma,\delta)}\right]\right|,
\end{align}
with the convention that $\Psi^{(1,1)}_{j,k} = \Psi^{(1,1)}$, and similarly for $\Psi^{(1,2)}_{j,k}$ and $\Psi^{(2,1)}_{j,k}$.

We now consider all possible cases of  $(\alpha,\beta,\gamma,\delta) \in \{1,2\}^4$ and $\# \{i,j,k,n,\ell\} \leq 3$.
Since $\Psi^{(1,1)}$ does not depend on any index and both $\Psi^{(1,2)}_{k}$ and $\Psi^{(2,1)}_{j}$ only depend on one index (not taking into account the dependence of $i$ since $\Psi^{(2,1)}_{i} = 0$ anyway), the number of cases to be considered considerably reduces for these terms.

In order to exploit this in the sequel, we introduce the following slightly abusive notation.
When considering the term  $\E_m[ \1_{B^m_i} \Psi_{j,k}^{(\alpha,\beta)} \Psi_{n,\ell}^{(\gamma,\delta)}]$ for fixed $\alpha, \beta, \gamma, \delta$, we
define the notion of \emph{relevant indices} to be the subset of indices $\{i,j,k,n,\ell\}$ appearing in this product after replacing $\Psi^{(1,1)}_{j,k}$ by $\Psi^{(1,1)}$ and similarly for $\Psi^{(1,2)}_{j,k}$, $\Psi^{(2,1)}_{j,k}$ and for the indices $n,\ell$.


To further reduce the number of cases that we have to consider, we next argue that we do not have to consider the cases $\{j,k,n,\ell\}$ with
 $J \cap \{j,k\} \cap \{n,\ell\} = \emptyset$, where $J$ is the set of relevant indices.
Indeed, in all these cases, the expectation factorizes, and we conclude by the bounds provided by Lemma \ref{lem:bounds.not.cross}.
In particular, we do not have to consider any case where $\Psi^{(1,1)}$ appears.

Moreover, if $j$ is a relevant index and $i=j$, then $\Psi^{(2,2)}_{j,k} = \Psi^{(2,1)}_{j} = 0$, and therefore, there is nothing to estimate. If \(j\) and \(k\) are both relevant indices and \(j=k\), then \(\Psi^{(2,2)}_{j,j}=0\), and therefore, there is nothing to estimate either.
The same reasoning applies to the cases where $i=n$ and \(n=\ell\), respectively.

We now list all the cases that are left to consider. Cases that are equivalent by symmetry we list only once.
We use the convention here, that we only specify which relevant indices coincide; relevant indices which are not explicitly denoted as equal are assumed to be different. The indices which are not relevant may take any number, in particular coinciding with each other or with relevant indices.

\begin{enumerate}
\item {$(\alpha,\beta,\gamma,\delta) = (2,2,2,2)$: }
Relevant indices: $\{i,j,k,n,\ell\}$. Since all the indices are relevant, we only have to consider cases where 
at least two pairs or three indices coincide. All the other cases are already covered when we have estimated $I^{i,j,k,n,\ell}$ with 
$\#\{i,j,k,n,\ell\} \geq 4$. The cases left to consider are
\begin{enumerate}
	\item $i=k$, $j=n$, \label{it:2222.ik.jn}
	\item $i=k$, $j=\ell$, \label{it:2222.ik.jl}
	\item $i=k=\ell$, \label{it:2222.ikl}
	\item $j=n$, $k=\ell$, \label{it:2222.jn.kl}
	\item $j=\ell$, $k=n$, \label{it:2222.jl.kn}
	\item $i=k=\ell$, $j=n$. \label{it:2222.ikl.jn}
\end{enumerate}

\item {$(\alpha,\beta,\gamma,\delta) = (2,1,2,2)$: }
Relevant indices: $\{i,j,n,\ell\}$. Cases to consider: 
\begin{enumerate}
	\item $j=n$, \label{it:2122.jn}
	\item $j=\ell$, \label{it:2122.jl}
	\item $i=\ell$, $j=n$.\label{it:2122.il.jn}
\end{enumerate}

\item {$(\alpha,\beta,\gamma,\delta) = (2,1,2,1)$: }
Relevant indices: $\{i,j,n\}$.  Only case to consider: $j=n$. \label{it:2121.jn}

\item {$(\alpha,\beta,\gamma,\delta) = (1,2,2,2)$: }
Relevant indices: $\{i,k,n,\ell\}$. Cases to consider: 
\begin{enumerate}
	\item $i=k=\ell$, \label{it:1222.ikl}
	\item $i=\ell$, $k=n$, \label{it:1222.il.kn}
	\item $k=n$. \label{it:1222.kn}
\end{enumerate}

\item {$(\alpha,\beta,\gamma,\delta) = (1,2,2,1)$: } \label{it:1221.kn}
Relevant indices: $\{i,k,n\}$. Only case to consider: $k=n$.

\item {$(\alpha,\beta,\gamma,\delta) = (1,2,1,2)$: }
Relevant indices: $\{i,k,\ell\}$.  Cases to consider:
\begin{enumerate}
	\item $k=\ell$, \label{it:1212.kl}
	\item $i = k = \ell$. \label{it:1212.ikl}
\end{enumerate}

\end{enumerate}

In order to conclude the proof of the lemma, it now remains to give estimates for the cases listed above.

\medskip

The case \eqref{it:2222.ik.jn}:
As mentioned at the beginning of Step 3.3, this is the case, where we rely on cancellations with $\Psi^{(2,1)}$ coming from case \eqref{it:2122.il.jn}.
We estimate
\begin{align}
	&\E_m\left[ \1_{B^m_i}(x) \Psi_{j,i}^{(2,2)}(x) (\Psi_{j\ell}^{(2,1)} - \Psi_{j\ell}^{(2,2)})(x) \right] \\
	&=  \iint  \1_{B^m(y_1)}(x) G  \left[\left(A\left[ \bigl((u)_{y_1}-v_1\bigr)  \delta^m_{y_1} \right]\right)_{y_2} \delta^m_{y_2} \right](x) G \left[\left(A\left( F-\Fcal\right)\right)_{y_2} \delta^m_{y_2} \right] (x) f(\de y_1,\de v_1) f(\de y_2,\de v_2) \\
	& =  \iint \rho(y_2)\1_{B^m(y_1)}(x) \left(A \left[ \bigl((u)_{y_1}-v_1\bigr)  \delta^m_{y_1}  \right]\right)_{y_2} (G \delta^m_{y_2})^2 (x)  \left(A (F - \Fcal) \right)_{y_2} f(\de y_1,\de v_1) \dd y_2.
\end{align}
Hence, since \(A\) maps \(L^2(\R^3)\cap\dot{H}^{-1}(\R^3)\) to \(L^{\infty}(\R^3)\) and by \eqref{eq:F-F.L^2}  
\begin{align}
	&\int \left|\E_m\left[ \1_{B^m_i} \Psi_{ji}^{(2,2)} (\Psi_{j\ell}^{(2,1)} - \Psi_{j\ell}^{(2,2)}) \right] \right| \dd x \\
	& \lesssim m^{-1}\iiint \rho(y_2)\1_{B^m(y_1)}(x) \left|\left(A \left[ \bigl((u)_{y_1}-v_1\bigr)  \delta^m_{y_1}  \right]\right)_{y_2}\right|  (G \delta^m_{y_2})^2 (x) f(\de y_1,\de v_1)\dd y_2 \dd x.
\end{align}

By \eqref{eq:G.delta.pointwise}
\begin{align} \label{eq:average.G.delta}
	\int \1_{B^m(y_1)}(x) (G \delta^m_{y_2})^2 (x) \dd x \lesssim m^{-3} \frac{1}{|y_2-y_1|^2 + m^{-2}}.
\end{align}
Combining this with the pointwise estimate \eqref{eq:A.delta.pointwise} yields
\begin{align}
	&\int \left|\E_m\left[ \1_{B^m_i} \Psi_{ji}^{(2,2)} (\Psi_{j\ell}^{(2,1)} - \Psi_{j\ell}^{(2,2)})\right] \right| \dd x \\
	& \lesssim m^{-4} \iint \rho(y_2) |(u)_{y_1}-v_1| \frac{1}{|y_2-y_1|^2 + m^{-2}} \left(1+\frac{1}{|y_2-y_1| + m^{-1}}\right)f(\de y_1,\de v_1) \dd y_2 \\
	& \lesssim m^{-4} \log m \int   |(u)_{y_1}-v_1| f(\de y_1,\de v_1) \lesssim m^{-4} \log m,
\end{align}
where we used \eqref{eq:V()} and \ref{ass:energy}.

\medskip

The case \eqref{it:2222.ik.jl} is similar. However, it turns out to be easier, since the singularity is subcritical, so we do not need to take into account cancellations. Indeed,
\begin{align}
	&  \E_m\left[ \1_{B^m_i}(x) \Psi_{ji}^{(2,2)}(x)  \Psi_{nj}^{(2,2)}(x)  \right]  \\
	&= \iint  \1_{B^m(y_1)}(x)   G  \left[\left(A\left[ \bigl( (u)_{y_1}-v_1\bigr) \delta^m_{y_1} \right]\right)_{y_2} \delta^m_{y_2} \right] (x)\\
	& \cdot G  \left[ \int \left(A  \left [ \bigl((u)_{y_2}-v_2\bigr) \delta^m_{y_2} \right]\right)_{y_3}  \delta^m_{y_3} f(\de y_3,\de v_3) ] \right]  (x) f(\de y_1,\de v_1) f(\de y_2,\de v_2)  \\
 	& = \iint \bigl( (u)_{y_1}-v_1\bigr)\bigl((u)_{y_2}-v_2\bigr) \1_{B^m(y_1)}(x) \left(A \delta^m_{y_1} \right)_{y_2}  (G  \delta^m_{y_2})(x)	\left(G \Rcal A \delta^m_{y_2}\right) (x) f(\de y_1,\de v_1) f(\de y_2,\de v_2) .
 \end{align}
 Thus, since \(G\Rcal\) maps \(L^2(K)\) to \(L^{\infty}(\R^3)\) and by \eqref{eq:G.delta.L^2_loc}
 \begin{align}
	&\int \left| \E_m\left[ \1_{B^m_i} \Psi_{ji}^{(2,2)}  \Psi_{nj}^{(2,2)}  \right] \right| \dd x  \\
 	&\lesssim \iint \bigl( (u)_{y_1}-v_1\bigr)\bigl((u)_{y_2}-v_2\bigr) \1_{B^m(y_1)}(x)  \left|\left(A \delta^m_{y_1}  \right)_{y_2}\right| \left| (G  \delta^m_{y_2})
 	 \right|(x)f(\de y_1,\de v_1) f(\de y_2,\de v_2). \label{eq:step:2222.ik.jl}
\end{align}
Now we proceed as in the previous case to estimate
\begin{align}
	&\int \left| \E_m\left[ \1_{B^m_i} \Psi_{ji}^{(2,2)}  \Psi_{nj}^{(2,2)}  \right] \right| \dd x  \\
	& \lesssim m^{-3} \iint \left( \bigl( (u)_{y_1}-v_1\bigr)^2+\bigl( (u)_{y_2}-v_2\bigr)^2 \right) \frac{1+\frac{1}{|y_2-y_1| + m^{-1}}}{|y_2-y_1| + m^{-1}} f(\de y_1,\de v_1) f(\de y_2,\de v_2) \\
	&\lesssim m^{-3}.
\end{align}

\medskip

The case \eqref{it:2222.ikl}:
We have 
\begin{align}
	& \E_m\left[ \1_{B^m_i}(x) \Psi_{ji}^{(2,2)}(x)  \Psi_{ni}^{(2,2)}(x) \right] \\ 
	& =  \int \1_{B^m(y_1)}(x) 
	\left( G \left[ \int \rho(y_2) \left(A \left[ \bigl((u)_{y_1}-v_1\bigr) \delta^m_{y_1}\right] \right)_{y_2} \delta^m_{y_2} \dd z \right] \right)(x)^2 f(\de y_1,\de v_1) \\
	& = \int \bigl((u)_{y_1}-v_1\bigr)^2 \1_{B^m(y_1)}(x) \left( G \Rcal A \delta^m_{y_1} \right)(x)^2 f(\de y_1,\de v_1).
\end{align}
Thus, using first that \(\|G\Rcal A \delta_{y_1}^m\|_{L^{\infty}(\R^3)} \lesssim 1\) as above, \ref{ass:energy} and \eqref{eq:()} together with \eqref{eq:V()}.
\begin{align}
	& \int \left| \E_m\left[ \1_{B^m_i} \Psi_{ji}^{(2,2)}  \Psi_{ni}^{(2,2)} \right] \right| \dd x \lesssim \int \int \bigl((u)_{y_1}-v_1\bigr)^2 \1_{B^m(y_1)}(x) f(\de y_1, \de v_1) \dd x \lesssim m^{-3}.
\end{align}

\medskip

The case \eqref{it:2222.jn.kl}: We compute
\begin{align}
	& \E_m\left[ \1_{B^m_i}(x) \Psi_{jk}^{(2,2)}(x)  \Psi_{jk}^{(2,2)}(x) \right] \\ 
	&= m^{-3} \iint (\rho)_x \rho(y_2)\left(G  \left[\left(A\left[ \bigl((u)_{y_1}-v_1\bigr) \delta^m_{y_1} \right]\right)_{y_2} \delta^m_{y_2} \right](x) \right)^2
 	 f(\de y_1,\de v_1) \dd y_2 \\
 	&= m^{-3} \iint (\rho)_x \rho(y_2)\bigl((u)_{y_1}-v_1\bigr)^2 \left(A \delta^m_{y_1} \right)^2_{y_2}  (G \delta^m_{y_2})^2(x)  f(\de y_1,\de v_1) \dd y_2.
\end{align}
Using \eqref{eq:G.delta.L^2_loc} twice, \eqref{eq:V()} together with \ref{ass:rho} and \ref{ass:energy}, we can successively estimate the integral in $x$, $y_2$ and $(y_1,v_1)$ to deduce
\begin{align}
	\int \left|  \E_m\left[ \1_{B^m_i} \Psi_{jk}^{(2,2)}  \Psi_{jk}^{(2,2)} \right]  \right| \dd x 
	& \lesssim m^{-3} \int  \rho(y_2)\bigl( (u)_{y_1}-v_1\bigr)^2 \left(A\left[  \delta^m_{y_1} \right]\right)_{y_2}^2
 	 f(\de y_1,\de v_1) \dd y_2 \\
 	 & \lesssim m^{-3} \int  \bigl( (u)_{y_1}-v_1\bigr)^2 f(\de y_1,\de v_1) 
 	 \lesssim m^{-3}.
\end{align}

\medskip
The case \eqref{it:2222.jl.kn}:
We just observe that by Young's inequality
\begin{align}
	&\int \left| \E_m\left[ \1_{B^m_i} \Psi_{jk}^{(2,2)}  \Psi_{kj}^{(2,2)} \right] \right| \dd x \leq \int \E_m\left[ \1_{B^m_i} \left( \left(\Psi_{jk}^{(2,2)}\right)^2 +  \left(\Psi_{kj}^{(2,2)}\right)^2 \right)\right] \dd x.
\end{align}
Thus, this case is reduced to case \eqref{it:2222.jn.kl}.

\medskip

The case \eqref{it:2222.ikl.jn}. Note that $\# \{i,j,k,n,\ell\} = 2$.
Hence, we only need a bound $m^{-2}$.
We have
\begin{align}
	& \E_m\left[ \1_{B^m_i}(x) \Psi_{ji}^{(2,2)}(x)  \Psi_{ji}^{(2,2)}(x) \right]\\
	&= \iint  \rho(y_2) \1_{B^m(y_1)}(x) \left(G  \left[\left(A \left[ \bigl((u)_{y_1}-v_1\bigr) \delta^m_{y_1} \right]\right)_{y_2} \delta^m_{y_2} \right](x) \right)^2 f(\de y_1,\de v_1) \dd y_2\\
	&= \iint  \rho(y_2) \bigl((u)_{y_1}-v_1\bigr)^2 \1_{B^m(y_1)}(x) \left(A\delta^m_{y_1}\right)^2_{y_2}  (G \delta^m_{y_2})^2(x) f(\de y_1,\de v_1) \dd y_2.
\end{align}
We can estimate the integral in $x$ using again \eqref{eq:average.G.delta}
\begin{align}
&\int \left| \E_m\left[ \1_{B^m_i} \Psi_{ji}^{(2,2)}  \Psi_{ji}^{(2,2)} \right]  \right| \dd x \\
&\leq \int \rho(y_2) \bigl((u)_{y_1}-v_1\bigr)^2 \1_{B^m(y_1)}(x) \left(A\delta^m_{y_1}\right)^2_{y_2}  (G \delta^m_{y_2})^2(x) f(\de y_1,\de v_1) \dd y_2 \dd x \\
& \lesssim  m^{-3} \int \rho(y_2) \bigl((u)_{y_1}-v_1\bigr)^2 \left(A\delta^m_{y_1}\right)^2_{y_2}   \frac{1}{|y_2-y_1|^2 + m^{-2}} f(\de y_1,\de v_1) \dd y_2.
\end{align}
Moreover, using \eqref{eq:A.delta.pointwise}, we find
\begin{align}
		&\int \left| \E_m \left[ \1_{B^m_i} \Psi_{ji}^{(2,2)}  \Psi_{ji}^{(2,2)} \right]  \right| \dd x  \\
		& \lesssim  m^{-3} \int \rho(y_2) \bigl((u)_{y_1}-v_1\bigr)^2\left( \frac{1}{|y_2-y_1|^2 + m^{-2}} + \frac{1}{|y_2-y_1|^4 + m^{-4}}\right)  f(\de y_1,\de v_1) \dd y_2 \\
		& \lesssim m^{-2}\int\bigl((u)_{y_1}-v_1\bigr)^2 f(\de y_1,\de v_1) \lesssim m^{-2},
\end{align}
where we used \eqref{eq:V()} and \ref{ass:energy} in the last estimate.
Note that this estimate is sufficient, since the number of different indices in this case is only $2$.

\medskip

The cases \eqref{it:2122.jn} and \eqref{it:2122.jl} are reduced  to the cases \eqref{it:2121.jn} and \eqref{it:2222.jn.kl} by Young's inequality, analogously as in the case \eqref{it:2222.jl.kn}.

\medskip

The case \eqref{it:2122.il.jn} was estimated together with the case \eqref{it:2222.ik.jn} if $k$ is different from the other indices.

If $k$ coincides with one of the other indices, the number of different indices is $2$ and we can reduce the case to the cases
\eqref{it:2121.jn} and \eqref{it:2222.ikl.jn} by Young's inequality.

\medskip

The case \eqref{it:2121.jn}:
In this case we get a factor $m^{-3}$ from $\1_{B^m_i}$ and thus the desired estimate follows from
\begin{align}
	\|\E_m[|\Psi^{(2,1)}_j|^2]\|_{L^1(K)}  \lesssim \int \rho(y_1) |(AF)_{y_1}|^2 \|G \delta_{y_1}^m\|_{L^2(K)}^2 \dd y_1 \lesssim 1,
\end{align}
where we used \eqref{eq:G.delta.L^2_loc} and \eqref{eq:V()}.

\medskip

The case \eqref{it:1222.ikl} is estimated by an analogous computation as the one at the end of the proof of Lemma \ref{lem:bounds.not.cross}, relying on the fact that
\begin{align}	\label{eq:Psi^1,2.infty}
	\|\Psi^{(1,2)}_{k}\|_{L^{\infty}(\R^3)} \lesssim |(u)_k-V_k|,
\end{align}
which is a direct consequence of \eqref{eq:G.delta.L^2_loc} and the fact that $G\rho$ is bounded from $L^2(K)$ to $L^\infty(\R^3)$.
Since the index $n$ is free, a similar bound can be used for $\Psi^{(2,2)}_{n,\ell}$. More precisely,
\begin{align}
	|\E_m[\1_{B^m_i}  \Psi^{(1,2)}_{i} \Psi^{(2,2)}_{n,i}]| &\leq \int \1_{B^m(y_1)} \left|G \Rcal A \left[\bigl((u)_{y_1}-v_1\bigr) \delta^m_{y_1}\right]\right|
	\left|G \rho A \left[\bigl((u)_{y_1}-v_1\bigr) \delta^m_{y_1}\right]\right| f(\de y_1,\de v_1)  \\
	&\lesssim \int  \1_{B^m(y_1)} |(u)_{y_1}-v_1|^2 f(\de y_1,\de v_1), 
\end{align}
since \( G\Rcal\) and \(G\rho\) map \(L^2(K)\) to \(L^{\infty}(\R^3)\) and using again \eqref{eq:G.delta.L^2_loc}. As before, integrating in $x$ yields a factor $m^{-3}$.

\medskip

The case \eqref{it:1222.il.kn}: Using \eqref{eq:Psi^1,2.infty} yields
\begin{align}
	|\E_m[\1_{B^m_i}  \Psi^{(1,2)}_{k} \Psi^{(2,2)}_{k,i}]| \lesssim \int \1_{B^m(y_1)} |(u)_{y_1}-v_1||(u)_{y_2}-v_2| G[\delta^m_{y_2}] |(A  \delta^m_{y_1})_{y_2}| f(\de y_1,\de v_1)f(\de y_2,\de v_2),
\end{align} 
which is the same as \eqref{eq:step:2222.ik.jl} which we have already estimated.

\medskip

The case \eqref{it:1222.kn} is reduced  to the cases \eqref{it:1212.kl} and \eqref{it:2222.jn.kl} by Young's inequality.

\medskip

The case \eqref{it:1221.kn} is reduced  to the cases \eqref{it:1212.kl} and \eqref{it:2121.jn} by Young's inequality.

\medskip

The cases  \eqref{it:1212.kl} and \eqref{it:1212.ikl} are estimated by an analogous computation as the one at the end of the proof of Lemma \ref{lem:bounds.not.cross}, relying on \eqref{eq:Psi^1,2.infty} again.

\appendix

\section{Appendix}

\subsection{Proofs of the auxiliary estimates from Section \ref{sec:Auxiliary}}\label{subsec:appendix_proofs}

\begin{proof}[Proof of Lemma \ref{lem:auxiliary.averages}]
	\textbf{(i)} Define
	\[[w](x) = \fint_{\partial B^m_x} w(y) \dd \mathcal{H}^2(y).\]
	We observe that for $w \in W^{1,p}(\R^3)$, \(1\leq p <\infty\)
	\begin{align}
	\|[w]\|_{L^p(\R^3)}^p = \int_{\R^3} \left |\fint_{\partial B^m(x)} w(y) \dd \mathcal{H}^2(y) \right|^p \dd x 
	&\leq \int_{\R^3} \fint_{\R^3} \1_{|x - y| = m^{-1}} |w(y)|^p \dd \mathcal{H}^2(y) \dd x\\
	&= \int_{\R^3} \fint_{\R^3} \1_{|y'| = m^{-1}} |w(y' + x)|^p \dd \mathcal{H}^2(y')  \dd x\\
	&= \int_{\R^3} \fint_{\R^3} \1_{|y'| = m^{-1}} |w(x')|^p \dd \mathcal{H}^2(y')  \dd x'\\
	&= \|w\|_{L^p(\R^3)}^p.
	\end{align}
By density,  the operator $[\cdot]$ is defined on $L^p(\R^3)$. Using an analogous argument also for the average $(\cdot)$ over the full ball yields \eqref{eq:()}.

\medskip

\textbf{(ii)} If \(w\in L^p(K)\), the fact that \(\rho\in L^{\infty}\) has compact support in $K$ implies \eqref{eq:V()}.

\medskip

\textbf{(iii)} To prove \eqref{eq:poincare.average.ball}, we first establish the following  inequality:

Let \(R>0\) and \(\varphi\in L^1(\R^3)\) with \(\varphi\geq 0\), \(\supp \varphi\subset B_R(0)\) and \(\|\varphi\|_{L^1}=1\). Let \(w\in \dot H^1(\R^3)\), then
\begin{align}\label{eq:Poincare.convolution}
\|\varphi\ast w - w\|_{L^2(\R^3)} \lesssim  R \|\nabla w\|_{L^2(\R^3)}.
\end{align}  

There are several ways to prove this. By scaling, it is enough to consider the case \(R=1\). We can use the Fourier transform: observe that \(\hat{\varphi}\in C^{\infty}(\R^3)\) with 
\[|\nabla\hat{\varphi}| = \left|\mathcal{F}(x\varphi)\right| \in L^{\infty}(\R^3).\]
Since \(\hat{\varphi}(0)=1\), this shows that there is a constant \(C>0\) such that \(|(1-\hat{\varphi})(k)| \leq C|k|\). Hence,
\begin{align*}
\|\varphi\ast w - w\|_{L^2(\R^3)}^2 & = \|(1-\hat{\varphi})\hat{w}\|_{L^2(\R^3)}^2 \leq C\|k\hat{w}\|_{L^2(\R^3)}^2 \leq C \|\nabla w \|_{L^2(\R^3)}^2.		
\end{align*}

Now,  \eqref{eq:poincare.average.ball} follows by choosing \(\varphi(x) = \1_{B^m(0)}(x)\).

\medskip
	
\textbf{(iv)} We note that $\Rcal w = [\rho (w)_\cdot]$. Thus, $\Rcal$ is a bounded operator from $L^2(K)$ to $L^2(\R^3) \cap \dot{H}^{-1}(\R^3)$ and from $H^1(K)$ to $H^1(\R^3)$ by the previous estimates, together with the assumption that $\rho \in W^{1,\infty}$ with compact support
and $L^{6/5}(\R^3) \subset \dot H^{-1}(\R^3)$.

\medskip

	For the estimate \eqref{eq:V-V.L^2}, we compute, for \(w\in \dot{H}^1(\R^3)\),
	\begin{align*}
	& \left\|  \Rcal w - \rho w \right\|_{L^2(\R^3)} \\
	& = \left\| \fint
	_{\partial B^m(x)} \rho(y) (w)_y \dd \mathcal{H}^2(y) - \rho(x) w(x) \right\|_{L^2(\R^3)} \\ 
	& \leq \left\| \fint_{\partial B^m(x)} \left(\rho(y)-\rho(x)\right) (w)_y \dd \mathcal{H}^2(y) \right\|_{L^2(\R^3)} + \left\|\fint_{\partial B^m(x)} \rho(x) \left((w)_y-w(x)\right)\dd \mathcal{H}^2(y) \right\|_{L^2(\R^3)} \\
	& =: J_1 + J_2.
	\end{align*} 
	Further, it is by Jensen's inequality
	\begin{align*}
	J_1^2 & = \int_{\R^3}\left| \fint_{\partial B^m(x)} \left(\rho(y)-\rho(x)\right) (w)_y \dd \mathcal{H}^2(y) \right|^2 \dd x 
	 \leq \int_{\R^3} \fint_{\partial B^m(x)} \left|\rho(y)-\rho(x)\right|^2 |(w)_y|^2 \dd \mathcal{H}^2(y) \dd x \\
	& \leq m^{-2} \|\nabla \rho\|_{L^{\infty}(\R^3)}^2 \|w\|_{L^2(\R^3)}^2,
	\end{align*}
	where we used \eqref{eq:()}.
	Moreover,
	\begin{align*}
	J_2^2 & = \int_{\R^3}\left| \fint_{\partial B^m(x)} \rho(x) \fint_{B^m(y)} w(z) - w(x) \dd z \dd y \right|^2 \dd x \\
	& \leq \|\rho\|^2_{L^{\infty}} \int_{\R^3}\left| \fint_{\partial B^m(x)} \fint_{B^m(y)} w(z) \dd z \dd y  - w(x) \right|^2 \dd x \\
	& =  \|\rho\|^2_{L^{\infty}} \int_{\R^3} \left| \int_{\R^3} \left(\fint_{\partial B^m(x)}   |B^m|^{-1} \1_{|y-z|\leq R_m} \dd y\right)(w(z)) \dd z - w(x) \right|^2 \dd x \\
	& = \|\rho\|^2_{L^{\infty}}\int_{\R^3} \left|\int_{\R^3}\varphi(x-z)w(z)\dd z - w(x)\right|^2 \dd x,
	\end{align*}
	with the choice
	\[ \varphi(x) = \fint_{\partial B^m(x)} |B^m|^{-1} \1_{|y|\leq R_m} \dd y.\]
	Using Fubini, we easily see that \(\varphi\) satisfies the assumptions to apply \eqref{eq:Poincare.convolution}. Hence
	\[J_2^2 \leq C m^{-2}\|\rho\|^2_{L^{\infty}(\R^3)} \|\nabla w\|_{L^2(\R^3)}^2.\]
	This proves \eqref{eq:V-V.L^2}. Finally, estimate \eqref{eq:V-V.H^-1} follows from testing with \(\psi \in \dot{H}^1(\R^3)\)
	\begin{align*}
	\langle \rho w - \Rcal w,\psi \rangle & = \langle w, \rho\psi - \Rcal \psi \rangle \leq m^{-1} \|w\|_{L^2(\R^3)} \|\rho\|_{W^{1,\infty}(\R^3)} \|\psi\|_{\dot H^1(\R^3)}.
	\end{align*}
	To justify the first line, observe that
	\begin{align*}
	\int_{\R^3} (\Rcal w) (x) \psi(x) \dd x & = \int \rho(x) (w)_x \fint_{\partial B^m(x)} \psi(y) \dd \mathcal{H}^2(y) \dd x \\
	& =  \int \rho(x) \left(\fint_{\R^3} \textbf{1}_{|x-z|\leq 1/m} w(z) \dd z\right) \fint_{\partial B^m(x)} \psi(y) \dd \mathcal{H}^2(y) \dd x \\
	& = \int_{\R^3} w(z) \left(\fint_{\R^3} \textbf{1}_{|x-z|\leq 1/m} \rho(x) \fint_{\partial B^m(x)} \psi(y) \dd \mathcal{H}^2(y) \dd x \right) \dd z \\
	& =  \int_{\R^3} w(z) (\Rcal \psi)(z) \dd z.
	\end{align*}
	
	\medskip
	
	\textbf{(v)} Recall that \(F=\rho u - j\). Since \(\rho \in L^{\infty}\) has compact support and \(u\in \dot{H}^1(\R^3)\), we have \(\rho u \in L^2(\R^3)\). Furthermore, from hypotheses \ref{ass:j} we have $j \in L^2(\R^3)$.
	Since \(\Fcal= \Rcal u - [j]\) and \(u\in L^2(K)\), we have \(\Fcal\in L^2(\R^3)\). Finally, we have with \(W_1= (u)_{1}-V_1\)
	\begin{align*}
		\E_m[W_1^2] & = \int_{\R^3\times\R^3} |(u)_x-v|^2 f(\de x,\de v)  \leq 2\int_{\R^3}\rho(x)|(u)_x|^2 \dd x + 2\int_{\R^3\times \R^3} |v|^2 f(\de x,\de v) \\
		& \leq C\|u\|_{L^2(K)} + 2\int_{\R^3\times \R^3} |v|^2 f(\de x,\de v)
	\end{align*}
	which is uniformly bounded by \eqref{eq:V()} and \ref{ass:energy}.
	
	To prove \eqref{eq:F-F.L^2}, we first focus on estimating the $L^2$-norm. Note that
	\begin{align*}
		F - \Fcal = \rho u - j - \left(\Rcal u - [j]\right).
	\end{align*}
	Hence, it is
	\begin{align*}
		\|F-\Fcal\|_{L^2(\R^3)} \leq \norm{\rho u -\Rcal u}_{L^2(\R^3)} + \norm{j-[j]}_{L^2(\R^3)}.
	\end{align*}
	Using \eqref{eq:V-V.L^2}, it is enough to see
	\begin{align*}
		\norm{w-[w]}_{L^2(\R^3)} & \lesssim m^{-1} \norm{w}_{\dot{H}^1(\R^3)} \quad \text{for all } w\in \dot{H}^1(\R^3).
	\end{align*}
	First, let \(w\in \mathcal{S}(\R^3)\). Then
	\begin{align*}
		\norm{w-[w]}_{L^2(\R^3)}^2 & = \int_{\R^3}\left|\fint_{\partial B^m(x)} w(x)-w(y) \dd \H^2(y)\right|^2 \dd x \\
		& \leq \int_{\R^3}\fint_{\partial B^m(x)} |w(x)-w(y)|^2 \dd \H^2(y) \dd x \\
		& \leq \int_{\R^3}\fint_{\partial B^m(x)} \int_0^1 |\nabla w(x+t(y-x))|^2|x-y|^2 \dd t \dd\H^2(y) \dd x \\
		& \lesssim m^{-2} \fint_{\partial B^m(x)} \int_0^1 \|\nabla w\|_{L^2(\R^3)}^2 \dd t \dd\H^2(y) \\
		& = m^{-2}\|w\|_{\dot{H}^1(\R^3)}^2,
	\end{align*}
	where we used Jensen's inequality twice and the fundamental theorem of calculus. Now by density of \(\mathcal{S}(\R^3)\) in \(\dot{H}^1(\R^3)\), we obtain the estimate of the $L^2$-norm in  \eqref{eq:F-F.L^2}. To estimate the $\dot H^{-1}$-norm in \eqref{eq:F-F.L^2}, we again argue by testing with \(\psi \in \dot{H}^1(\R^3)\). By \eqref{eq:V-V.H^-1}, it is enough to see
	\begin{align*}
		|\langle j-[j],\psi \rangle| = |\langle j, \psi - [\psi] \rangle| \leq \|j\|_{L^2(\R^3)}\|\psi-[\psi]\|_{L^2(\R^3)} \leq m^{-1}\|\psi\|_{\dot{H}^1(\R^3)}. 
	\end{align*}
	This finishes the proof.
\end{proof}

\begin{proof}[Proof of Lemma \ref{lem:Gdelta}]
%
%
	Recalling the definition of $B^m(y) = B_{R_m}(y)$ and \eqref{eq:R_m.Poisson}, it is well-known that
		\begin{align}\label{eq:GdeltaStokes}
		G\delta^m_y(x) = \begin{cases}
			m \mathrm{Id} & x  \in B^m(y) \\
			g(x-y) - \frac{R_m^2}{6} \Delta g(x-y) & x\in \R^3\setminus B^m(y),
		\end{cases}
	\end{align}
	with $g$ as in \eqref{eq:fund.sol}. Then \eqref{eq:G.delta.pointwise}, \eqref{eq:nablaG.delta.pointwise} and  \eqref{eq:G.delta.L^2_loc} follow immediately. \eqref{eq:nablaG.delta.pointwise} implies that \(\|G\delta\|_{\dot{H}^1(\R^3)}\lesssim m^{1/2}\) and, since  \(G\) is an isometry from \(\dot{H}^{-1}(\R^3)\) to \(\dot{H}^1(\R^3)\), this proves \eqref{est:dirac.-1}. The bounds for \(A\) follow by using the identity \(A=G-A\rho G\)
	and that $A\rho$ maps $L^2_\loc(\R^3)$ to $L^\infty(\R^3)$
\end{proof}

\begin{proof}[Proof of Lemma \ref{lem:delta-delta_m}]
    To deduce the bound for $G \delta_y$ in $H^s_\loc(\R^3)$, note for example that $e^{-|x-y|} G \delta_y = e^{-|x-y|}/(4 \pi |x-y|) \in H^s(\R^3) $ (e.g. by Fourier).
	The corresponding estimate for $A$ follows from the identity $A = G - A \rho G$ (cf. \eqref{eq:relation.A.G}) and the fact that $A \rho$ maps $H^s_\loc$ to $H^s_\loc$.
	
	For the second estimate, observe that \(H^{3/2+\eps}(K')\) embeds into the space of \(\eps\)-Hölder continuous functions on \(K'\). Hence, we may estimate, for every \(w\in H^{3/2+\eps}(K')\)
	\begin{align*}
	\langle \delta_y^m - \delta_y, w \rangle \leq \fint_{B^m(y)} |w(x)-w(y)| \dd H^{2}(x) \leq m^{-\eps} \|w\|_{C^{\eps}(K')} \leq Cm^{-\eps}\|w\|_{H^{3/2+\eps}(K')}.
	\end{align*}
	This concludes the proof.
\end{proof}

\begin{proof}[Proof of Lemma \ref{lem:G-1Gm}]
	By \eqref{eq:g^m.Stokes}, $G - G^m$ is a convolution operator with convolution kernel
	\begin{align}
		\bar g_m := \eta_m g - \psi_m.
	\end{align}
	Thus,  to prove \eqref{eq:G-G^m.H^k.to.H^k} and  \eqref{eq:G-1GmHktoHk+1}
	it suffices to show
	\begin{align} \label{eq:bar.g_m.L^1}
		\|\nabla^l \bar g_m\|_{L^1(\R^3)} \leq m^{-2 + l}
	\end{align}
	for $l = 0,1$. Moreover, \eqref{eq:bar.g_m.L^1} for $l=2$ implies that $G^m$ is a bounded operator from $\dot H^l(\R^3)$ to $\dot H^{l+2}(\R^3)$ since we know that $G$ is a bounded operator from $\dot H^l(\R^3)$ to $\dot H^{l+2}(\R^3)$.
	
	By definition of $\eta_m$, we have for all $l \in \N$
	\begin{align}
		|\nabla^l (\eta_m g)| \lesssim m^{1 + l} \1_{B_{3 R_m}(0) \setminus B_{2 R_m}(0)}.
	\end{align}
	In particular, for all $1 \leq p \leq \infty$ and all $l \in \N$
	\begin{align} \label{eq:L^p.eta_m.g}
		\|\nabla^l (\eta_m g)\|_{L^p(\R^3)} \lesssim m^{1 + l-3/p}.
	\end{align}
	In view of \eqref{eq:Bogovski.estimate}, this implies
	\begin{align} \label{eq:L^p.psi}
		\|\nabla^l (\eta_m g)\|_{L^p(\R^3)} \lesssim m^{1 + l-3/p},
	\end{align}
	for all $l \geq 1$ and all $1 < p < \infty$. By the H\"older inequality,
	this bound also holds for $p = 1$ and by the Poicar\'e inequality also for $l = 0$.
	Combining \eqref{eq:L^p.eta_m.g} and \eqref{eq:L^p.psi} yields \eqref{eq:bar.g_m.L^1}.
\end{proof}

\subsection{Estimates for \texorpdfstring{$\mathbf{\mathfrak{J}_3}$}{J3}}
In this part of the appendix, we detail the estimates of \(\J_3\) from \eqref{def:frak.J_2}
We follow the same strategy as for \(\J_1\) described in Steps 3.2 and 3.3 of the proof of Lemma \ref{lem:variance.in.balls}. Therefore, we just name and prove the relevant lemmas. Observe that we need weaker bounds. If we want to show \(|\mathfrak{J}_3|\lesssim m^{-2}\), this requires
	\[I_{3,\nabla}^{i,j,k,l} =  \E_m\left[\int_{B_i} \nabla \Psi_{j,k}(x) \nabla\Psi_{n,\ell}(x) \dd x\right] \lesssim m^{-\alpha+2},\quad\alpha = \#\{i,j,k,n,\ell\}.\]
	
As before, we write \(\nabla \Psi_{j,l} = \nabla\Psi^{(1)}_k + \nabla \Psi^{(2)}_{j,l}\), where
\begin{align}
\nabla\Psi_{k}^{(1)} := \nabla\Psi^{(1,1)} + \nabla\Psi_{k}^{(1,2)} :=\nabla G\left[\rho A\left(F - W_k\delta_{k}^m\right) \right],
\end{align}
and 
\begin{align}
\nabla\Psi_{j,k}^{(2)} := \nabla\Psi_{j}^{(2,1)} +  \nabla\Psi_{j,k}^{(2,2)}:=(1-\delta_{ij})\nabla G\left[ \bigg(A\left(F - W_k\delta_{k}^m\right) \bigg)_j \delta_{j}^m \right].
\end{align}
Recall that \(W_k=(u)_k-V_k\) and \(F=\rho u - j\).

We observe that
\begin{equation}
\label{eq:exp.VAVu.grad}
	\begin{aligned}
		\E_m[\nabla\Psi^{(1,1)}] &= \nabla G  \rho  A F, \\
		\E_m[\nabla\Psi^{(1,2)}_k] &= \nabla G  \rho  A  \Fcal, \\
		\E_m[\nabla\Psi^{(2,1)}_j] &= (1- \delta_{ij}) \nabla G  \Rcal  A F, \\
		\E_m[\nabla\Psi^{(2,2)}_{j,k}] &= (1- \delta_{ij})(1-\delta_{jk}) \nabla G  \Rcal A\Fcal.
	\end{aligned}
\end{equation}
Furthermore, we observe that the only difference to the discussion of \(\mathfrak{J}_1\) is that the outmost \(G\) is replaced by \(\nabla G\). Hence, we we will apply the same strategy as before using the analogous auxiliary estimates for the gradient.	
	
	 We start by giving the corresponding lemmas in the case \(\#\{i,j,k,n,\ell\}\geq 4\).

\begin{lem} \label{lem:cancellations.not.cross_grad}
	\begin{align}
	\left\|\E_m[\nabla\Psi_{k}^{(1)}] \right\|_{L^2(\R^3)} &\lesssim m^{-1}, \label{eq:Psi^1.grad} \\
	\left\|\E_m[\nabla\Psi_{j,k}^{(2)}] \right\|_{L^2(\R^3)} &\lesssim m^{-1} \quad \text{if } j \neq k \label{eq:Psi^2.grad}.
	\end{align}
\end{lem}

\begin{lem} \label{lem:cancellations.cross_grad}
	\begin{align}
	\label{eq:jn.grad} 
	\|\E_m\left[\nabla\Psi_{j,k} \nabla\Psi_{j,\ell}\right] \|_{L^1(\R^3)}  \lesssim m^{-1} \quad \text{if } \#\{i,j,k,\ell\} = 4, \\
	\label{eq:kl.grad} 
	\|\E_m\left[\nabla\Psi_{j,k} \nabla\Psi_{n,k}\right]\|_{L^1(\R^3)}  \lesssim m^{-1} \quad \text{if } \#\{i,j,k,n\} = 4, \\
	\label{eq:jl.grad} 
	\|\E_m\left[\nabla\Psi_{j,k} \nabla\Psi_{n,j}\right]\|_{L^1(\R^3)} \lesssim m^{-1} \quad \text{if } \#\{i,j,k,n\} = 4.
	\end{align}
\end{lem}


\begin{lem} \label{lem:bounds.not.cross_grad}
	We have for any $i,j,k$
	\begin{align}
	& \left\|\E_m[\nabla\Psi^{(1,1)}] \right\|_{L^2(\R^3)} 
	+ \left\|\E_m[\nabla\Psi_{k}^{(1,2)}] \right\|_{L^2(\R^3)} \\
	& \qquad \qquad + \left\|\E_m[\nabla\Psi_{j}^{(2,1)}] \right\|_{L^2(\R^3)} 
	+ \left\|\E_m[\nabla\Psi_{j,k}^{(2,2)}] \right\|_{L^2(\R^3)}\lesssim m. \label{eq:bounds.separate.no.i.grad}\\	
	& \left\|\E_m[\1_{B^m_i}\nabla\Psi^{(1,1)}] \right\|_{L^2(\R^3)} 
	+ \left\|\E_m[\1_{B^m_i}\nabla\Psi_{k}^{(1,2)}] \right\|_{L^2(\R^3)} \\
	& \qquad \qquad + \left\|\E_m[\1_{B^m_i}\nabla\Psi_{j}^{(2,1)}] \right\|_{L^2(\R^3)} 
	+ \left\|\E_m[\1_{B^m_i}\nabla\Psi_{j,k}^{(2,2)}] \right\|_{L^2(\R^3)} \lesssim m^{-5/2}. \label{eq:bounds.separate.with.i.grad}\\
	\end{align}
\end{lem}

\begin{proof}[Proof of Lemma \ref{lem:cancellations.not.cross_grad}]
	By \eqref{eq:exp.VAVu.grad}, we have
	\begin{align*}
		\E_m[\Psi^{(1)}_{k}] = \nabla G \rho A(F -  \Fcal).
	\end{align*}
	Using \eqref{eq:F-F.L^2} yields \eqref{eq:Psi^1.grad}.
	
	Similarly, for $j \neq k$, $i \neq j$,
	\begin{align*}
	\E_m[\Psi^{(2)}_{j,k}]  = \nabla G \Rcal A  (F -  \Fcal).
	\end{align*}
	Using again \eqref{eq:F-F.L^2}  yields \eqref{eq:Psi^2.grad}.
\end{proof}

\begin{proof}[Proof of Lemma \ref{lem:cancellations.cross_grad}]
	Regarding \eqref{eq:jn.grad}, we have
	\begin{align}
		 \E_m\left[\nabla\Psi_{j,k} \nabla\Psi_{j,\ell}\right] 
		& =\int \rho(y_1) \left( \nabla G\rho A(F-\Fcal) - \left(A(F-\Fcal)\right)_{y_1} \nabla G \delta_{y_1}^m \right)^2 \dd y_1,
	\end{align}
	and hence
	\begin{align}
		& \norm{\E_m\left[\nabla\Psi_{j,k} \nabla\Psi_{j,\ell}\right]}_{L^1(\R^3)}  \\
		& \lesssim \|\nabla G\rho A(F-\Fcal)\|_{L^2(\R^3)}^2 + \int \rho(y_1)\left(A(F-\Fcal)\right)_{y_1} \|\nabla G \delta_{y_1}^m\|_{L^2(K)}^2 \dd y_1 \\
		& \lesssim m^{-2} + m^{-1}\\
		& \lesssim m^{-1}, 
	\end{align}
	where we used \eqref{eq:F-F.L^2} for both terms and \eqref{est:dirac.-1} for the second term.
	
	Regarding \eqref{eq:kl.grad}, we compute
	\begin{align}
		& \E_m\left[\nabla \Psi_{j,k}\nabla\Psi_{n,k}\right]  = \int \rho(y_1)\left(\nabla G(\rho-\Rcal)AF - \bigl((u)_{y_2}-v_2\bigr)\nabla G(\rho-\Rcal)A\delta_{y_2}^m \right)^2 f(\de y_2,\de v_2).
	\end{align}
	Hence, we obtain
	\begin{align}
		&\norm{\E_m\left[\nabla \Psi_{j,k}\nabla\Psi_{n,k}\right]}_{L^1(\R^3)} \\
		& \lesssim \norm{\nabla G (\rho-\Rcal)A \rho u}_{L^2(\R^3)}^2 + \sup\limits_{y_1} \|\nabla G (\rho-\Rcal)A \delta_{y_1}^m\|_{L^2(\R^3)}^2 \int \bigl((u)_{y_2}-v_2\bigr)^2 f(\de y_2,\de v_2) \\
		& \lesssim m^{-1},
	\end{align}
	where we used \eqref{eq:V-V.H^-1} for both terms and \eqref{eq:G.delta.L^2_loc} and \ref{ass:energy} for the second term. Finally, \eqref{eq:jl.grad} follows from \eqref{eq:jn.grad} and \eqref{eq:kl.grad} via Young's inequality.
\end{proof}

\begin{proof}[Proof of Lemma \ref{lem:bounds.not.cross_grad}]
	The first estimate, \eqref{eq:bounds.separate.no.i.grad}, follows directly from \eqref{eq:exp.VAVu.grad} and \eqref{est_W_i.F} together with the fact that the operators \(\nabla G \rho A \), \(\nabla G\rho A \), \(\nabla G\Rcal A \) and \(\nabla G \Rcal A\) are all bounded operators from \(\dot{H}^{-1}(\R^3)\) to \( \dot H^1(\R^3)\). 
	
	Regarding \eqref{eq:bounds.separate.with.i.grad}, these estimates follow from \eqref{eq:bounds.separate.no.i.grad} if \(i\neq k\). If \(i=k\), we only need to consider those terms, in which \(k\) appears, i.e. \(\nabla\Psi_k^{(1,2)}\) and \(\nabla\Psi^{(2,2)}_{j,k}\). Again, we only need to consider the case \(j\neq k = i\). 
	
	Then
	\begin{align}
		\norm{\E_m\left[\1_{B_i^m} \nabla\Psi_{j,i}^{(1,2)} \right]} & = \norm{\int \1_{B^m_{y_1}} \nabla G \rho A[\bigl((u)_{y_1}-v_1\bigr) \delta^m_{y_1}]  f(\de y_1,\de v_1)} \\
		& \leq \sup\limits_{y_1\in\R^3} \|\nabla G\rho A\delta_{y_1}^m \|_{L^\infty(\R^3)}  \left\|\int  \bigl((u)_{y_1}-v_1\bigr) \1_{B^m_{y_1}} f(\de y_1,\de v_1) \right\|_{L^2(\R^3)} \\
		& \lesssim m^{-5/2}.
	\end{align}
	Here, we used \eqref{est:dirac.-1} and that $G\rho$ maps $\dot H^1(\R^3)$ to $W^{1,\infty}(\R^3)$ for the first term, and \ref{ass:energy} as well as´ \eqref{eq:()} followed by \eqref{eq:V()} for the second.
	Since for $j \neq i$, 
	\begin{align}
		 \E_m [\1_{B^m_i} \nabla\Psi^{(2,2)}_{j,i}] & = \int \1_{B^m(y_1)} \nabla G \Rcal A  \left[ \bigl((u)_{y_1}-v_1\bigr) \delta^m_{y_1}\right] f(\de y_1,\de v_1),
	\end{align}
	the estimate of this term is analogous.
\end{proof}

This finishes the cases in which at most \(2\) indices are equal. For the remaining cases, we can again follow the same strategy as for \(\J_1\).  We provide here only the necessary estimates. All the other estimates follow by applying Young's inequality and reducing the proofs to the estimates given here, just as in the proof for \(\J_1\).

\begin{lem}\label{lem:bounds.cancellations.grad}
	The corresponding estimates in the case \((\alpha,\beta,\gamma,\delta) = (2,2,2,2)\) are:
	\begin{align}
		i=k, j=n: && \int\left| \E_m\left[\1_{B_i^m} \nabla\Psi_{j,i}^{(2,2)}(\nabla\Psi_{j}^{(2,1)} - \nabla\Psi^{(2,2)}_{j,\ell}) \right] \right| \dd x & \lesssim m^{-2}. \label{eq:2222.ik.jn.grad}\\
		i=k, j=\ell: && \int\left| \E_m\left[\1_{B_i^m} \nabla\Psi_{j,i}^{(2,2)}\nabla\Psi^{(2,2)}_{n,j} \right] \right| \dd x & \lesssim m^{-2}. \label{eq:2222.ik.jl.grad} \\
		i=k=\ell: && \int\left| \E_m\left[\1_{B_i^m} \nabla\Psi_{j,i}^{(2,2)}\nabla\Psi^{(2,2)}_{n,i} \right] \right| \dd x & \lesssim m^{-2}. \label{eq:2222.ikl.grad}\\
		j=n, k=\ell: && \int\left| \E_m\left[\1_{B_i^m} \left|\nabla\Psi_{j,k}^{(2,2)}\right|^2  \right] \right| \dd x & \lesssim m^{-2}. \label{eq:2222.jn.kl.grad}\\
		i=k=\ell, j=n: && \int\left| \E_m\left[\1_{B_i^m} \left|\nabla\Psi^{(2,2)}_{j,i} \right|^2 \right] \right| \dd x & \lesssim 1. \label{eq:2222.ikl.jn.grad}\\
	\end{align}
	The corresponding estimate in the case \((\alpha,\beta,\gamma,\delta) = (2,1,2,1)\) is:
	\begin{align}
		j=n: && \int \left| \E_m\left[\1_{B_i^m}\left|\nabla\Psi_j^{(2,1)}\right|^2\right] \right| \dd x & \lesssim m^{-2}. \label{eq:2121.jn.grad}
	\end{align}
	The corresponding estimates in the case \((\alpha,\beta,\gamma,\delta) = (1,2,2,2)\) are:
	\begin{align}
		i=k=\ell: && \int\left| \E_m\left[\1_{B_i^m} \nabla\Psi_{i}^{(1,2)}\nabla\Psi^{(2,2)}_{n,i} \right] \right| \dd x & \lesssim m^{-2}. \label{eq:1222.ikl.grad}\\
		i=\ell, k = n: && \int\left| \E_m\left[\1_{B_i^m} \nabla\Psi_{k}^{(1,2)}\nabla\Psi^{(2,2)}_{k,i} \right] \right| \dd x & \lesssim m^{-1}. \label{eq:1222.il.kn.grad}
	\end{align}
\end{lem}

\begin{proof}[Proof of Lemma \ref{lem:bounds.cancellations.grad}]
	For \eqref{eq:2222.ik.jn.grad}, it is
	\begin{align}
		& \E_m\left[\1_{B_i^m} \nabla\Psi_{j,i}^{(2,2)}(\nabla\Psi_{j,\ell}^{(2,1)} - \nabla\Psi^{(2,2)}_{j,\ell}) \right] \\
		& = \iint \rho(y_2)  \1_{B^m_{y_1}}(x) \left(A\left[\bigl((u)_{y_1}-v_1\bigr) \delta_{y_1}^m\right]\right)_{y_2} (\nabla G \delta_{y_2}^m)^2(x) \left(A(F-\Fcal)\right)_{y_2} f(\de y_1,\de v_1) \de y_2.
	\end{align}
	By \eqref{eq:nablaG.delta.pointwise}, it holds
	\begin{align}\label{eq:average.nablaG.delta}
		\int \1_{B^m_{y_1}}(x) (\nabla G \delta_{y_2}^m)^2(x) \dd x \lesssim m^{-3} \frac{1}{|y_2-y_1|^4+m^{-4}},
	\end{align}
	and thus analogously as in the corresponding term for $\J_1$ using \eqref{eq:A.delta.pointwise}
	\begin{align}
		& \int\left| \E_m\left[\1_{B_i^m} \nabla\Psi_{j,i}^{(2,2)}(\nabla\Psi_{j,\ell}^{(2,1)} - \nabla\Psi^{(2,2)}_{j,\ell}) \right] \right| \dd x \\
		& \lesssim m^{-4} \iint  \rho(y_2)|(u)_{y_1}-v_1| \frac{1}{|y_2-y_1|^4 + m^{-4}}\left(1+\frac{1}{|y_2-y_1|+m^{-1}}\right) f(\de y_1,\de v_1) \dd y_2\\
		& \lesssim m^{-2}.
	\end{align}
	Regarding \eqref{eq:2222.ik.jl.grad}, we compute
	\begin{align*}
		& \E_m\left[ \1_{B^m_i}(x) \nabla\Psi_{ji}^{(2,2)}(x)  \nabla\Psi_{nj}^{(2,2)}(x)  \right] \\
		& = \iint \bigr((u)_{y_1}-v_1\bigl) \bigr((u)_{y_2}-v_2\bigl)\1_{B^m(y_1)}(x) \left(A \delta^m_{y_1} \right)_{y_2}  (\nabla G  \delta^m_{y_2})(x)	\left(\nabla G \Rcal A  \delta^m_{y_2}\right) (x) f(\de y_1,\de v_1) f(\de y_2,\de v_2).
	\end{align*}
	Now we use that $G\Rcal$ maps $\dot H^1(\R^3)$ to $W^{1,\infty}(\R^3)$  to deduce as in the previous case
	\begin{align}
		& \int\left|\E_m\left[ \1_{B^m_i}(x) \nabla\Psi_{ji}^{(2,2)}(x)  \nabla\Psi_{nj}^{(2,2)}(x)  \right]\right| \dd x\\
		& \lesssim m^{1/2}m^{-3}\int \left(\bigl((u)_{y_1}-v_1\bigr)^2 + \bigl((u)_{y_1}-v_1\bigr)^2\right) \frac{1+\frac{1}{|y_2-y_1|+m^{-1}}}{|y_2-y_1|^2+m^{-2}} f(\de y_1,\de v_1) f(\de y_2,\de v_2)\\
		& \lesssim m^{-5/2}\log m.
	\end{align}
	For \eqref{eq:2222.ikl.grad}, we get
	\begin{align}
	& \E_m\left[ \1_{B^m_i}(x) \nabla\Psi_{ji}^{(2,2)}(x)  \nabla\Psi_{ni}^{(2,2)}(x) \right] = \int \bigl((u)_{y_1}-v_1\bigr)^2 \1_{B^m_{y_1}}(x) 
	\left( \nabla G \Rcal A\delta^m_{y_1} \right)(x)^2 f(\de y_1,\de v_1).
	\end{align}
	Thus by \eqref{eq:G.delta.L^2_loc}, \eqref{eq:V()} and \ref{ass:energy}, it is
	\begin{align}
		& \int\left|\E_m\left[ \1_{B^m_i}(x) \nabla\Psi_{ji}^{(2,2)}(x)  \nabla\Psi_{ni}^{(2,2)}(x) \right]\right|\dd x \lesssim m^{-2}.
	\end{align}
	The case \eqref{eq:2222.jn.kl.grad}:
	\begin{align}
	& \E_m\left[ \1_{B^m_i}(x) \nabla\Psi_{jk}^{(2,2)}(x)  \nabla\Psi_{jk}^{(2,2)}(x) \right] \\ 
	&= m^{-3} \iint (\rho)_x \rho(y_2) \bigl((u)_{y_1}-v_1\bigr)^2 \left(A \delta^m_{y_1} \right)^2_{y_2}  (\nabla G \delta^m_{y_2})^2(x) f(\de y_1,\de v_1) \dd y_2.
	\end{align}
	Using \eqref{est:dirac.-1},\eqref{eq:G.delta.L^2_loc}, \eqref{eq:V()} and \ref{ass:energy}, we get
	\begin{align}
	&\int \left|  \E_m\left[ \1_{B^m_i} \Psi_{jk}^{(2,2)}  \Psi_{jk}^{(2,2)} \right]  \right| \dd x \lesssim m^{-2} \int  \rho(y_2) \bigl((u)_{y_1}-v_1\bigr)^2 \left(A \delta^m_{y_1} \right)^2_{y_2}   f(\de y_1,\de v_1) \dd y_2 \lesssim m^{-2}.
	\end{align}
	For the next estimate \eqref{eq:2222.ikl.jn.grad}, we get
	\begin{align}
	& \E_m\left[ \1_{B^m_i}(x) \nabla\Psi_{ji}^{(2,2)}(x)  \nabla\Psi_{ji}^{(2,2)}(x) \right]\\
	&= \iint  \rho(y_2) \bigl((u)_{y_1}-v_1\bigr)^2 \1_{B^m_{y_1}}(x) \left(A \delta^m_{y_1}\right)^2_{y_2}  (\nabla G \delta^m_{y_2})^2(x) f(\de y_1,\de v_1) \dd y_2.
	\end{align}
	By using again \eqref{eq:average.nablaG.delta} and \eqref{eq:A.delta.pointwise}, we get
	\begin{align}
	&\int \left| \E_m \left[ \1_{B^m_i} \nabla\Psi_{ji}^{(2,2)}  \nabla\Psi_{ji}^{(2,2)} \right]  \right| \dd x  \\
	& \lesssim  m^{-3} \int \rho(y_2) \bigl((u)_{y_1}-v_1\bigr)^2  \left( \frac{1}{|y_2-y_1|^4 + m^{-4}} + \frac{1}{|y_2-y_1|^6 + m^{-6}}\right)  f(\de y_1,\de v_1) \dd y_2 \\
	& \lesssim \int \bigl((u)_{y_1}-v_1\bigr)^2f(\de y_1,\de v_1)\lesssim 1.
	\end{align}
	To estimate \eqref{eq:2121.jn.grad}, observe
	\begin{align}
		& \E_m\left[\1_{B_i^m}\left|\nabla\Psi_j^{(2,1)}\right|^2\right]  \lesssim m^{-3} \int(\rho)_x \rho(y_1)(AF)_{y_1}^2 |\nabla G \delta_{y_1}^m|(x)^2 \dd y_1,
	\end{align}
	and hence by \eqref{est:dirac.-1}, it holds
	\begin{align}
		\int \E_m\left[\1_{B_i^m}\left|\nabla\Psi_j^{(2,1)}\right|^2\right]\dd x & \lesssim m^{-2} \int \rho(y_1) (AVu)_{y_1}^2 \dd y_1  \lesssim m^{-2}.
	\end{align}
	For \eqref{eq:1222.ikl.grad}, it holds
	\begin{align*}
	 |\E_m\left[\1_{B_i^m} \nabla\Psi_{i}^{(1,2)}\nabla\Psi^{(2,2)}_{n,i} \right]|  &\leq \int  \1_{B^m_{y_1}} \left|\nabla G \Rcal A \left[\bigl((u)_{y_1}-v_1\bigr) \delta^m_{y_1}\right]\right| \left| \nabla G\rho A\left[\bigl((u)_{y_1}-v_1\bigr)\delta^m_{y_1}\right] \right| f(\de y_1,\de v_1) \\
		&  \lesssim m \int \1_{B^m_{y_1}} \bigl((u)_{y_1}-v_1\bigr)^2 f(\de y_1,\de v_1),
	\end{align*}
	where we used \eqref{est:dirac.-1}. Thus
	\begin{align}
		\int |\E_m\left[\1_{B_i^m} \nabla\Psi_{i}^{(1,2)}\nabla\Psi^{(2,2)}_{n,i} \right]|  \dd x \lesssim m^{-2}.
	\end{align}
	Finally for \eqref{eq:1222.il.kn.grad}, it is
	\begin{align}
		 \left| \E_m\left[\1_{B_i^m} \nabla\Psi_{k}^{(1,2)}\nabla\Psi^{(2,2)}_{k,i} \right] \right| &\leq \int \1_{B^m_{y_1}} \rho(y_1) \bigl|(u)_{y_2}-v_2\bigr|^2 \left|\nabla G \rho A \delta_{y_2}^m \right| \left|(A\delta_{y_1}^m)_{y_2}\right| \left|\nabla G \delta^m_{y_2}\right| \dd y_1 f(\de y_2,\de v_2) \\
		& \lesssim m^{1/2} \int \1_{B^m_{y_1}}\rho(y_1) \bigl|(u)_{y_1}-v_1\bigl|^2 \left|(A\delta_{y_1}^m)_{y_2}\right| \left|\nabla G \delta^m_{y_2}\right| \dd y_1 f(\de y_2, \de v_2),
	\end{align}
	where we used \eqref{eq:G.delta.L^2_loc}. This is estimated as in \eqref{eq:2222.ik.jl.grad} to get
	\begin{align}
		\int \left| \E_m\left[\1_{B_i^m} \nabla\Psi_{k}^{(1,2)}\nabla\Psi^{(2,2)}_{k,i} \right] \right| \dd x \lesssim m^{-1}.
	\end{align}
	This finishes the proof.
\end{proof}

\section*{Acknowledgements}

The authors thank Juan J.L. Vel\'azquez for valuable discussions.
The authors have been supported by the Deutsche Forschungsgemeinschaft (DFG, German Research Foundation)
through the collaborative research center ``The Mathematics of Emerging Effects'' (CRC 1060, Projekt-ID 211504053)
and the Hausdorff Center for Mathematics (GZ 2047/1, Projekt-ID 390685813).

\subsection*{Conflict of interest} The authors declare that they have no conflict of interest.

\subsection*{Data availability}
Data sharing not applicable to this article as no datasets were generated or analysed during the current study.


\printbibliography
\end{document}